\documentclass[DIV=12,parskip=half]{scrartcl}

\usepackage{style}

\title{\(\LL\alpha\)-Regularization of the Beckmann Problem}
\author{
    Dirk Lorenz
    \thanks{%
	Institute of Analysis and Algebra,
	TU Braunschweig,
	38092 Braunschweig, Germany,
	(\texttt{d.lorenz@tu-braunschweig.de, h.mahler@tu-braunschweig.de})
    }
    \and
    Hinrich Mahler\footnotemark[1]
    \and
    Christian Meyer
    \thanks{
	Fakultät für Mathematik,
	TU Dortmund,
	44221 Dortmund, Germany
	(\texttt{christian2.meyer@tu-dortmund.de})
    }
}

\begin{document}

\maketitle
\begin{abstract}
    We investigate the problem of optimal transport in the so-called
        Beckmann form, \ie given two Radon measures on a compact set,
        we seek an optimal flow field which is a vector valued Radon measure on
        the same set that describes a flow between these two measures and
        minimizes a certain linear cost function.

    We consider \(\LL\alpha\) regularization of the problem, which guarantees uniqueness and
        forces
        the solution to be an integrable function rather than a Radon measure.
    This regularization naturally gives rise to a semi-smooth Newton scheme that can be used to
        solve the problem numerically.
    Besides motivating and developing the numerical scheme, we also include approximation results
        for vanishing regularization in the continuous setting.
\end{abstract}

\section{Introduction}

The Beckmann formulation of optimal transport is the problem of finding a flow field that
    describes how to move some measure onto another measure of the same mass such that a certain
    linear cost functional is minimal.
It was first introduced in~\cite{Beckmann1952} in a more general form.
Specifically, for a domain \(\Omega\subset\RR^d\), two Radon measures \(\mu^+,\mu^-\) on
    \(\Omega\) with \(\mu^+(\Omega) = \mu^-(\Omega)\) and a continuous cost function \(w:\Omega\to
    [0,\infty)\) our goal is to solve 
    \begin{equation}\tag{BP}\label{eq:beckmann}
        \inf_{\substack{q\in\M*\,,\\\div[] q = \mu}} \intO w \dabsq\,,
    \end{equation}
    where we abbreviated \(\mu:=\mu^+-\mu^-\) and the divergence constraint has to be understood
    in a
    suitable weak sense.
Existence of solutions is well known~\cite{Santambrogio2015,Dweik_2018}, but since the objective
    functional in~\cref{eq:beckmann} is not strictly convex, solutions may not be unique.
Moreover, for general Radon measures \(\mu^+,\mu^-\), a solution may not admit a density \wrt the
    Lebesgue measure.
Hence, standard approximation tools from numerical analysis are not applicable.
This motivates the use of regularization of the continuous problem to obtain approximate solutions
    that are functions instead of measures, which in turn can be treated by classical
    discretization
    techniques in order to solve the regularized problem.
Here, we aim to employ \(\LL\alpha\)-regularization which, as we will see, also naturally gives
    rise to a semi-smooth Newton scheme that can be used to solve the problem numerically.

The Beckmann problem is closely related to other problems of optimal transport theory, namely the
    so called Monge problem and the Kantorovich problem as well as the Monge-Kantorovich
    equation~\cite{Santambrogio2015,Ambrosio2003}.
For example, for \(w\equiv 1\)~\cref{eq:beckmann}
    is equivalent to the Kantorovich problem (with Euclidian cost),~\cite[§
        4.2.1]{Santambrogio2015}.

\subsection{Notation and problem statement}

Before we formulate our problem, let us fix the notation that will be used in the remainder.
The space of Radon measures and the set of probability measures on \(\Omega\subset\RR^d\) will be
    denoted by \(\M\) and \(\P\), respectively.
The space of vector valued Radon measures will be denoted by \(\M*\) and we will use the same
    convention for all other classes of measures and functions as well.
With \(\CC\) and \(\CC[k]\) we denote the spaces of continuous functions and \(k\) times
    continuously differentiable functions, respectively.

For a Banach space \(X\) we will denote its topological dual by \(X^*\).
The \(d\)-dimensional Lebesgue measure will be denoted by \(\leb^d\) and, where appropriate,
    integrals \wrt the Lebesgue measure are simply denoted by \(\dx\) with the appropriate
    integration
    variable \(x\).
For a set \(\Omega\subset\RR^d\) we will also use the shorthand notation \(\abs{\Omega} :=
    \leb^d(\Omega)\).
For the space of \(p\)-integrable functions on \(\Omega\) with respect to the Lebesgue measure,
    the symbol \(\Lp\)	will be used.
The symbol \(\W<k,p>\)	denotes the Sobolev space of functions for which the weak derivatives up
    to order \(k\) are functions in \(\Lp\).

When a measure \(\nu\) is absolutely continuous with respect to another measure \(\mu\), written
    as \(\nu\ll\mu\), the Radon-Nikodym derivative of \(\nu\) \wrt to \(\mu\), \ie the density of
    \(\nu\) \wrt \(\mu\), will be denoted by \(\tfrac\dnu\dmu\).
Conversely, by \(\I: \L1* \to \M*\) we denote the embedding, which identifies an integrable
    function with a Radon measure on \(\Omega\) via
    \[
    (\I(f))(A) := \int_A f \dleb\qquad \forall A\subset\Omega\,.
\]
Hence, \(\I(\tfrac\dnu\dleb) = \nu\).

With slight abuse of notation, we will denote the Nemytskii-operator
    \(%
    q \mapsto (\Omega\ni x\mapsto F(x,q(x)))%
    \)
    associated with a function \(F: \Omega\times\RR^d\to \RR^l\) by the same symbol.
The characteristic function of a set \(A\) will be denoted by \(\1_A\).
In contrast, \(\ind_A\) denotes the indicator functional of \(A\).
We denote the Euclidian norm on \(\RR^d\) with \(\abs{\blank}\) and the positive part  of a scalar
    \(c\) as \(c_+ := \max\{c, 0\}\).
The inner product of \(x,y\in\RR^d\) will be denoted by \(x\cdot y\).

In the following we will consider a compact domain \(\Omega\).
For \(f:\Omega\to\RR\) and
    \(c\in\RR\), we will use the shorthand notation
    \begin{equation*}
        \{f > c\} := \set{x\in\Omega}{f(x)> 0}
    \end{equation*}
    and analogously for \( \{f \geq c\}\), \( \{f < c\}\), \( \{f \leq c\}\) and \(\{f\neq c\}\).

The regularized Beckmann problem of optimal transport considered in this work now reads as
    \begin{equation}\tag{BP\(_\epsilon\)}\label{eq:reg_beckmann}
    \inf_{\substack{q\in\La*,\,\\\div[] q = \mu}} \intO w \abs{q}\dleb + \frac\epsilon\alpha
    \Lanorm{q}^\alpha\,.
\end{equation}%

Let us summarize our standing assumptions:
    \begin{assumption}\label{assmpt:general}
    We assume that \(d\in\N\) and \(\Omega\subset\RR^d\) is a compact set, whose interior is a
    bounded Lipschitz domain in the sense of~\cite[Chapter 1.2]{Grisvard2011}.
The cost function \(w:\Omega\to\RR\) is continuous.
Assume \(1<\alpha<\tfrac d{d-1}\).
Finally, we assume \(\mu_i\in\P\), for \(i=1,2\).
\end{assumption}

\begin{remark}\label{remark:embedding}
    \begin{enumerate}
        \item In contrast to standnard notation in PDE literature, we use the symbol \(\Omega\)
              for a closed set.
              Nevertheless, for convenience, we simply write \(\W<k,p>\) instead of
                  \(\W<k,p>[\interior(\Omega)]\) for Sobolev spaces.
              \item\label{remark:embedding:one}
              Note that by standard Sobolev embeddings
                  (e.g.~\cite[Theorem 4.12]{Adams2003}), it
                  holds that \(\W<1,\alpha'>\embed \CC\), since \(\alpha < \tfrac d{d-1}\).
              Hence, \(\M \embed {(\W<1,\alpha'>)}^*\).
              This allows us to use arbitrary measures \(\mu^+,\mu^-\in\P\) as marginals
                  in~\cref{eq:reg_beckmann}.

        \item Note that for the integral \(\intO w \abs{q}\dleb\) to exist, the cost function
              \(w\) does not need to be continuous and the problem may be formulated for more
              general cost
              functions.
              However, some of the results in this work require this assumption and for simplicity
                  it shall be assumed throughout the paper.
    \end{enumerate}
\end{remark}

\subsection{Related Work}\label{sec:lit-review}

Due to its relation with other optimal transport problems, the Beckmann problem has been
    considered in a number of different settings.

The authors of~\cite{solomon:2014} tackle the Beckmann problem with uniform cost function \(w\)
    from a geometry processing point of view to compute the distances between points on discrete
    surfaces.
The Helmholtz-Hodge decomposition and the spectral decomposition of the Laplacian are used to
    reformulate the Beckmann problem into an unconstrained problem, where the coefficients of the
    spectral decomposition are the optimization variables.
The authors then pass to a discrete setting and truncate the spectral decomposition, which reduces
    the problem size and gives an approximation of the original problem.

Several publications employ first order schemes to solve the Beckmann problem.
In~\cite{li:2017}, the authors discretize the problem via a finite differences scheme and employ
    the Chambolle-Pock algorithm.
They, too, only consider uniform cost \(w\), which allows to derive closed form expressions for
    the involved proximal operators.
To ensure uniqueness, they add a regularization term similar to to the one
    in~\cref{eq:reg_beckmann}, but only consider the case \(\alpha=2\).
The methods of~\cite{li:2017} are extended to unbalanced transport (\ie \(\mu^+(\Omega) \neq
    \mu^-(\Omega)\)) in~\cite{ryu:2017} and~\cite{liu:2018} proposes a multilevel initialization
    approach to speed up the computation time for fine grids.
Another first order scheme is covered in~\cite{jacobs:2019}, where a variant of the Chambolle-Pock
    algorithm is analyzed, which involves the computation of optimal step sizes.
The results are applied to an ROF formulation of the Beckmann problem in two dimensions with
    uniform weight and without regularization.
Moreover, an estimate for the error in the objective value is derived.
In~\cite{benamou:2015} multiple different problems are covered, including the Beckmann problem
    with general cost or \(\LL p\)-regularization (in the context of so-called congested
    transport),
    but not both at the same time.
The problems are solved numerically by solving the dual formulation
    by the ADMM algorithm.
This requires to solve a Laplace equation with Neumann boundary conditions
    in each iteration step.

The authors of~\cite{benmansour:2009,benmansour:2010} consider the closely related problem of
    traffic congestion~\cite{Carlier_2008}.
This problem generalizes the Beckmann problem by allowing
    the cost function \(w\) to depend on \(q\) in the sense \(w = w(x,\abs{q(x)})\) and the
    so-called
    traffic intensity is computed instead of \(q\), which allows to model a congestion effect.
A fast
    marching algorithm is proposed to treat the problem numerically.
In~\cite{Brasco_2010} the authors consider regularity results for this line of work and model the
    congestion by a term \(\tfrac 1p \abs{.
}^p\).
This corresponds to our regularization term, however
    they only consider uniform cost.~\cite{hatchi:2017,Brasco2013} consider a even more general,
    anisotropic setting and~\cite{hatchi:2017} includes numerical examples, which rely
    on~\cite{benamou:2015}.

A different type of regularization is employed in~\cite{barrett:2007}, where the authors use the
    Monge-Kantorovich equation as starting point and consider the functional \(\intO
    w\bd\abs{q}^r\)
    with \(r>1\) after smoothing the marginals \(\mu^+\) and \(\mu^-\) accordingly.
After providing a convergence result for \(r\to0\), the authors switch to a discrete setting and
    give another approximation result for increasing discretization fineness.
The numerical scheme then relies on a fixed-point iteration of the form \(\abs{x_i}^{r-2}(x_{i+1}
    - x_i) + \abs{x_i}^{r-2}x_i\), where an additional regularization is required due to the
    non-smoothness of \(\abs{\blank}\).
We point out that in contrast to~\cref{eq:reg_beckmann} this choice of regularization does not
    preserve the non-smooth structure of~\cref{eq:beckmann}.
The setting of~\cite{barrett:2007} is extended to a setting of unbalanced transport
    in~\cite{barrett:2009}.

The authors of~\cite{Facca_2020} propose a dynamic formulation of the Monge-Kantorovich equations
    (for uniform cost) and conjecture that the solution approximates the solution of the static
    equations for \(t\to\infty\).
However, the conjecture is still open.
The authors argue that the
    dynamic formulation naturally adds a regularization to the problem and derive an Euler scheme
    for
    solving the problem numerically.

\subsection{Organization}

The remainder of this work is organized as follows.
We start in~\cref{sec:theory} by rigorously defining the divergence constraint
    in~\cref{eq:reg_beckmann} and proving existence and uniqueness of solutions.
Afterwards we derive a semi-smooth Newton iteration in~\cref{sec:ssn}, which will also involve a
    second regularization.
We detail how to choose appropriate step sizes via an auxiliary minimization problem and make a
    connection between that problem and~\cref{eq:reg_beckmann}.
\Cref{sec:approx} is concerned with approximation results.
More precisely, we prove weak convergence of minimizers of the regularized problems towards
    minimizers of~\cref{eq:reg_beckmann,eq:beckmann} under suitable assumptions.
After discussing numerical examples in~\cref{sec:numerics}, we finally conclude
    in~\cref{sec:conclusion}.

\section{Existence of solutions}\label{sec:theory}

Let us rigorously define the divergence constraint in problem~\cref{eq:reg_beckmann}.
Motivated by the zero-flux boundary condition, the divergence constraint in~\cref{eq:beckmann} is
    to be understood as
    \begin{equation}\label{eq:mdiv_weak}
    -\intO \grad \phi \cdot\dq = \intO \phi\dmu
    \quad \forall\, \phi \in \CC[1][\Omega]\,.
\end{equation}
Therefore, the equality constraint in the regularized problem~\cref{eq:reg_beckmann} reads
    \begin{equation}\label{eq:cont_weak1}
    -\intO q\cdot \grad \phi \dleb = \intO \phi\dmu\,.
\quad \forall\, \phi \in \W<1,\alpha'>
\end{equation}

\begin{lemma}
    Let \(q\in \La*\) and let~\cref{assmpt:general} hold.
    Then \(q\) solves~\cref{eq:cont_weak1} if and only if it solves
        \begin{equation}\label{eq:cont_weak2}
            -\intO q\cdot \grad \phi \dleb = \intO \phi\dmu
            \quad \forall\, \phi \in \Wempty,
        \end{equation}
        where
        \begin{equation*}
        \Wempty
        := \set{ v\in  \W<1,\alpha'>}{\intO v(x) \dx = 0}\,.
    \end{equation*}
\end{lemma}
\begin{proof}
    If \(q\) solves~\cref{eq:cont_weak1}, then it trivially also solves~\cref{eq:cont_weak2}.
    On the other hand, if \(q\) solves \(\cref{eq:cont_weak2}\), then for every \(\phi\in
        \Wempty\)
        and every \(c\in \RR\), the assumptions on the marginals imply
        \begin{equation*}
        \begin{aligned}
        -\intO q(x)\cdot \grad (\phi(x) + c) \dx
        &= \intO \phi\dmu
        = \intO \phi\dmu + c (\mu^+(\Omega) - \mu^-(\Omega))\\
        &= \intO (\phi(x) + c)\dmu(x).
    \end{aligned}
    \end{equation*}
    Since \(\W<1,\alpha'> = \Wempty + \RR\), this gives the assertion.
\end{proof}

Using the previous result, we can now define the divergence on \(\La*\) as follows.

\begin{definition}\label{def:div}
    Define
        \begin{equation*}
            \begin{aligned}
                 & \div : \La* \to \Wperp := \set{v\in {(\W<1,\alpha'>)}^*}{\dual{v}{1} = 0},
                \\
                 & \dual{\div q}{\phi} := - \intO q\cdot \grad \phi\dleb \quad \forall \phi\in
                \Wempty\,,
            \end{aligned}
        \end{equation*}
        where \(\grad\) denotes the usual weak gradient.
\end{definition}

\begin{remark}\label{rem:mu_in_wperp}
    Recalling~\cref{remark:embedding}, we
        observe that \(\mu\in \Wperp\), since
        \(\dual{\mu}{1} = \mu^+(\Omega) - \mu^-(\Omega)  = 0\).
    Thus,~\cref{eq:cont_weak2} (and~\cref{eq:cont_weak1}, respectively) is equivalent to
        \begin{equation*}
        \div q = \mu \quad \text{in } \Wperp.
    \end{equation*}
\end{remark}

Next, we give a characterization of \({(\Wempty)}^*\).

\begin{lemma}
    The space \(\Wperp\) is isomorphic to \({(\Wempty)}^*\).
\end{lemma}
\begin{proof}
    On the one hand, it is clear that a functional in \(\Wperp\) defines a functional on
        \(\Wempty\) so that \(\Wperp \subset {(\Wempty)}^*\).

    On the other hand, the Hahn-Banach theorem implies that every \(\ell \in  {(\Wempty)}^*\) can
        be
        extended to a functional \(L\) on \(\W<1,\alpha'>\).
    If we define \(\bar v := \abs{\Omega}^{-1} \intO v(x)\dx\), then we observe for the functional
        \(L\) that
        \begin{equation*}
        \dual{\ell}{v - \bar v} = \dual{L}{v} - \bar v \dual{L}{1} \quad \forall\, v\in
        \W<1,\alpha'>\,.
    \end{equation*}
    If we now define \(\hat L \in {(\W<1,\alpha'>)}^*\) by \(\hat L(v) := L(v) - \bar v L(1)\),
        then

        \(\hat L(1) = 0\), \ie, \(\hat L \in \Wperp\),
        and \(\ell(v) = \hat L(v)\) for all \(v\in \Wempty\).
\end{proof}

\begin{remark}
    In complete analogoy to the above argumentation, see that \(q\in\M*\)
        solves~\cref{eq:mdiv_weak} if and only if \(q\) solves
        \begin{equation*}
        -\intO \grad \phi \cdot\dq = \intO \phi\dmu
        \quad \forall\, \phi \in \CCempty := \left\{ v\in  \CC[1]\,\middle|\,\intO v(x) \dx =
        0\right\}\,.
    \end{equation*}
    We can then define
        \begin{equation*}
            \begin{aligned}
                 & \mdiv : \M* \to \MMperp := \set{v\in {(\CC[1])}^*}{\dual{v}{1} = 0},
                \\
                 & \dual{\mdiv q}{\phi} := - \intO \grad \phi\cdot\dq \quad \forall \phi\in
                \CCempty
            \end{aligned}
        \end{equation*}
        and obtain that \(\MMperp\) is isomorphic to \({(\CCempty)}^*\).
    Hence, the divergence constraint in~\cref{eq:beckmann} can be understood as
        \[
        \mdiv q = \mu \quad \text{in }\MMperp\,.
    \]
    Note that clearly \(\CCempty \embed \Wempty\).
    Hence, for \(q\in\M*\) with \(q\ll\leb^d\) and
        \(\dens q\in\La*\), \(\mdiv q = \mu\) in \(\MMperp\) immediately implies \(\div \dens q =
        \mu\) in
        \(\Wperp\).
\end{remark}

The following two corollaries follow directly from the above definitions.

\begin{corollary}
    The adjoint operator \(\div^*\) of \(\div:\La* \to \Wperp\) is given by \(-\grad : \Wempty \to
        \Lap*\).
\end{corollary}

\begin{corollary}
    \begin{enumerate}
        \item The divergence operator \(\mdiv\) is continuous \wrt \ws convergence in \(\M*\).
        \item The divergence operator \(\div\) is continuous \wrt weak convergence in \(\La*\).
    \end{enumerate}
\end{corollary}

Before proving existence and uniqueness of solutions for~\cref{eq:reg_beckmann}, we cover
    surjectivity of the divergence operator under suitable assumptions.

\begin{assumption}\label{assmpt:laplace}
    Assume that \(\Omega\) is
        such that the equation
        \begin{equation}\label{eq:poisson}
            \div[\LL{\alpha'}] \grad y = \nu
            \quad\text{in }\W<-1,\alpha'><\perp>
        \end{equation}
        has a unique solution \(y\in\Wempty\) for every \(\nu\in \W<-1,\alpha'><\perp>\).
    Note that the associated solution operator, denoted by \(\laplace_{\alpha'}^{-1}:
        \W<-1,\alpha'><\perp> \to \Wempty\)
        is continuous by the open mapping theorem.
\end{assumption}
\begin{remark}
    Note that~\cref{assmpt:laplace} holds in two and three dimensions provided that the interior
        of \(\Omega\) is
        a
        bounded Lipschitz domain in the spirit of~\cite[Chapter 1.2]{Grisvard2011}.
    See
        e.g.~\cite[Theorem 3]{Groeger1989}
        for \(d=2\) and~\cite[Theorem~1.6]{Zanger2000} for \(d=3\).
    We will assume~\cref{assmpt:laplace} to hold for the remainder of this work.
\end{remark}

\begin{lemma}\label{thm:div:surjective}
    Let~\cref{assmpt:laplace} hold.
    Then, the divergence operator \(\div\) is surjective.
\end{lemma}
\begin{proof}
    We denote the solution operator of~\cref{eq:poisson} as \(\laplace_{\alpha'}\inv\).
    By identifying \(\W<1,\alpha><\emptyset>\) with its bi-dual space, we
        note that the adjoint operator
        \(
        {({\laplace_{\alpha'}}^{-1})}^*: \Wperp\to \W<1,\alpha><\emptyset>
        \)
        is continuous as well with \(\norm{{({\laplace_{\alpha'}}^{-1})}^*} \leq
        \norm{{\laplace_{\alpha'}}^{-1}}\).
    Moreover, we observe, that
        \({({\laplace_{\alpha'}}^{-1})}^* = {({\laplace_{\alpha'}}^*)}^{-1}\) and
        \begin{equation*}
        \laplace_{\alpha'}^* = \div \grad : \W<1,\alpha><\emptyset>\to\Wperp\,.
    \end{equation*}
    Hence, the elliptic equation
        \begin{equation*}
            \intO \grad y \cdot \grad \psi \dleb = \sca \psi \bd \nu\qquad \forall\psi\in\Wempty
        \end{equation*}
        has a unique solution \(y\in \Wa<\emptyset>\) for all \(\nu\in\Wperp\).
    By setting \(q = -\grad y\), we find \(\div q = \nu\) in \(\Wperp\), which shows the
        surjectivity of \(\div\) from
        \(\La*\) to \(\Wperp\).
\end{proof}

\begin{remark}
    Due to~\cref{remark:embedding}~\labelcref{remark:embedding:one},~\cref{thm:div:surjective}
        also implies the surjectivity of \(\mdiv\).
\end{remark}

Finally, we obtain an existence result.

\begin{corollary}
    Let~\cref{assmpt:general,assmpt:laplace} hold.
    For ever \(\epsilon>0\) there is a unique solution for problem~\cref{eq:reg_beckmann}.
\end{corollary}
\begin{proof}
    First note that due to~\cref{rem:mu_in_wperp} it holds \(\mu\in \Wperp\), so that
        by~\cref{thm:div:surjective} the feasible set is non-empty.

    Let now \((q_n)\subset\M\) be a minimizing sequence.
    Without loss of generality we assume that each \(q_n\) is feasible and due to the
        regularization term \((\dens{q_n})\) is bounded in \(\La\).
    We can thus extract a weakly convergent subsequence (denoted by the same symbol) with weak
        limit \(\bar q\in\La*\).
    As \(w\in\CC\embed\Lap\), the objective functional is clearly lower semi continuous in \(\La\)
        and thus, \(\bar q\) is a solution to~\cref{eq:reg_beckmann}.

    Uniqueness of the solution follows trivially from the strict convexity of
        \(\norm{\blank}_{\La}\).
\end{proof}

\section{Semi-Smooth Newton}\label{sec:ssn}

We first derive the first order optimality system for~\cref{eq:reg_beckmann}.

\begin{proposition}
    There exists a Lagrange multiplier \(y\in\Wempty\) such that the solution \(q\)
        of~\cref{eq:reg_beckmann} fulfills
        \begin{align}
            \epsilon\abs{q}^{\alpha-2} q + \partial \abs{q}_{1,w} + \grad y & \ni 0 \qquad\text{in
            }
            \Lap*
            \label{eq:gradeq}
            \\
            \div q                                                          & = \mu \qquad\text{in
            } \Wperp\,, \label{eq:constr}
        \end{align}
        where \(\abs{q}_{1,w}(x) := w(x)\abs{q(x)}\).
\end{proposition}
\begin{proof}
    Let us denote \(C := \{ q\in \La*  : \div q = \mu \}\) such that~\cref{eq:reg_beckmann} is
        equivalent to
        \begin{equation*}
        \inf_{q\in\La*} \intO w \abs{q}\dleb + \frac\epsilon\alpha \Lanorm{q}^\alpha + \ind_C(q).
    \end{equation*}
    Since the first two addends of the objective are continuous on the whole space \(\La*\)
        and \(C\) is nonempty due to~\cref{thm:div:surjective}, the sum rule for convex
        subdifferentials
        is applicable, which gives that the solution \(q\) of~\cref{eq:reg_beckmann} satisfies
        \begin{align*}
            0          & \in	\epsilon\abs{q}^{\alpha-2} q + \partial \abs{q}_{1,w} + \partial
            \ind_C(q)                                                                        \\
            \qquad \Leftrightarrow \quad
            \exists\xi & \in \partial \abs{q}_{1,w}, \quad
            \intO (\epsilon\abs{q}^{\alpha-2} q + \xi)(p - q) \dleb\geq 0 \quad \forall p \in C
            \\
            \qquad \Leftrightarrow \quad
            \exists\xi & \in \partial \abs{q}_{1,w}, \quad
            \epsilon\abs{q}^{\alpha-2} q + \xi \in {\ker(\div)}^\perp = \textup{ran}(\div^*),
        \end{align*}
        where we employed~\cite[§6.6, Theorem~2]{luenberger},
        which holds due to
        the surjectivity of \(\div\) by~\cref{thm:div:surjective}.
    Since \(\div^* = - \grad\), this gives the assertion.
\end{proof}
We observe that the multi-valued map
    \[
        (x,q) \mapsto \epsilon \abs{q}^{\alpha-2} q + w(x)\partial\abs{q}
    \]
    has a single-valued inverse, which we denote by
    \begin{equation}\label{eq:defF}
    \F: \Omega\times\RR^d\to\RR\,,\quad\F(x,p) =
        {\Big( \frac{1}{\epsilon} \big(\abs{p} - w(x)\big)\Big)}_+^{\alpha'-1}
    \frac{p}{\abs{p}}\,.
\end{equation}
Since~\cref{eq:gradeq} is a pointwise equation (as identity in \(\Lap*\)), this yields
    that~\cref{eq:gradeq}--\cref{eq:constr} are equivalent to
    \begin{equation}\label{eq:y-condition}
    \div \F(-\grad y) = \mu \quad \text{in } \Wperp\,.
\end{equation}
where the Nemytskii-operator \(\F\) maps \(\Lap*\) to \(\La*\).
By~\cref{def:div}, the weak form of~\cref{eq:y-condition} is given by
    \begin{equation}\label{eq:ypde}
    - \intO \F(-\grad y) \cdot \grad \phi\dleb
    = \intO \phi\dmu
    \quad \forall\, \phi \in \Wempty\,.
\end{equation}

We can now formally write down a semi-smooth Newton iteration for solving~\cref{eq:ypde} as
    follows.
\begin{algorithm}[H]
\caption{Semi-Smooth Newton Iteration for solving~\cref{eq:ypde}}\label{alg:ssn}
\begin{algorithmic}
\Require{} \(y\in\Wempty\)
\For{\(k=1,\ldots\)}
\State{}
Choose a step size \(\sigma_{k}>0\)
    \State{} find \(\eta\in\Wempty\) such that \(\forall\, \phi \in \Wempty\)
    \begin{equation}\label{eq:ssn:step}
        \intO \of{\Deriv_p\F(-\grad y) \grad \eta} \cdot \grad \phi\dleb
        =
        \intO \F(-\grad y) \cdot \grad \phi\dleb + \intO \phi\dmu
    \end{equation}
    \State{} Update \(y \gets y + \sigma_k\eta\)
    \EndFor{}
    \end{algorithmic}
    \end{algorithm}

    \begin{remark}
    We emphasize that~\cref{eq:ssn:step} is purely formal.
For~\cref{alg:ssn} to converge, we would need for \(\F\) to be Newton-differentiable from
    \(\Lap*\) to \(\La*\) and existence of solutions to~\cref{eq:ssn:step} in the appropriate
    spaces.
While the latter issue will be resolved by an additional Huber-regularization,
    see~\cref{eq:ypde:reg} below, the Newton-differentiability probably requires an additional
    smoothing step, as applied for instance in~\cite[Section 6.1]{Ulbrich2002}.
This is subject to future research.
\end{remark}

Due to the positive part in~\cref{eq:defF}, \(\F(x,p)\) has vanishing slope for \(\abs{p}\leq w\),
    which will clearly lead to illposedness of the Newton step~\cref{eq:ssn:step}.
As mentioned above, to overcome this
    issue, we introduce a Huber type regularization term~\cite{Huber1964} \(\R\) of the form
    \begin{equation*}
        \R: \Omega\times\RR^d\to\RR^d,\quad
        \R(x,p) = \frac{\delta p}{\max(\abs{p},w)}\,,
    \end{equation*}
    where \(\delta>0\) is a regularization parameter.
Denoting \(\G:=\F + \R\)
    we thus replace~\cref{eq:ypde} by
    \begin{equation}\label{eq:ypde:reg}
    - \intO \G(-\grad y) \cdot \grad \phi\dleb
    = \intO \phi\dmu
    \quad \forall\, \phi \in \Wempty.
\end{equation}
and~\cref{eq:ssn:step} by
\begin{equation}\label{eq:ssn:step:reg}
    \begin{split}
        \intO \Deriv_q\G(-\grad y) &\grad \eta \cdot \grad \phi\dleb
        = \\
        & \intO \G(-\grad y) \cdot \grad \phi\dleb + \intO \phi\dmu
        \quad \forall\, \phi \in \Wempty\,.
    \end{split}
\end{equation}

\subsection{Step Size Rule}

In order to apply Armijo bracktracking, we lift~\cref{eq:y-condition} to a minimization problem.
To that end, we observe that both \(\F\) and \(\R\) admit an antiderivative (\wrt \(p\)), namely
    \begin{align*}
    \calF: \Omega\times\RR^d\to\RR^d\,,\quad\calF(x, p) &:=
    \frac{\epsilon}{\alpha'}
    {\Big(\frac{1}{\epsilon}\,\max\big\{\abs{p} - w, 0\big\}\Big)}^{\alpha'}\text{
        and}%
    \\
    \calR: \Omega\times\RR^d\to\RR^d\,,\quad\calR(x,p) &:= \delta\, \max\big\{\abs{p}, w\big\} +
    \frac\delta 2 \,\min\Big\{\frac{\abs{p}^2}{w}, w\Big\}\,.
\end{align*}
More precisely, we obtain the following result.

\begin{lemma}
    Both \(\calF,\calR:\Lap*\to \L1*\) are Gateaux-differentiable with Gateaux-derivatives given
        by
        \begin{align*}
            \bd \calF(p;\psi) & = \F(p)\cdot \psi\,, \\
            \bd \calR(p;\psi) & = \R(p)\cdot \psi
        \end{align*}
        respectively, where \(\F,\R:\Lap*\to \La*\).
\end{lemma}
\begin{proof}
    Let now \(p,\psi\in \Lap*\).
    Elementary calculations show
        \[
            \lim_{t\to 0} \frac {\calF(x, (p+t\psi)(x)) - \calF(x, p(x))}t \to	\F(x, p(x)) \cdot
            \psi(x)
        \]
        for a.e.
    \(x\in\Omega\) and similarly for \(\calR\).
    By Lebesgue's dominated convergence
        theorem, it suffices to show that the right hand side is a function in \(\L1*\) and the
        mapping
        \(\psi\mapsto \F(p)\cdot \psi\) is continuous.

    To that end, note that by Hölder's inequality
        \begin{align*}
        \intO \abs{ \F(x, p(x)) \cdot \psi(x)} \dx
        & \leq \epsilon^{1-\alpha'} \norm{\psi}_{\Lap} \norm{{(\abs{p}-w)}_+^{\alpha'-1}}_{\La}\\
        & \leq \epsilon^{1-\alpha'} \norm{\psi}_{\Lap*} \norm{p}_{\Lap*} < \infty\,.
    \end{align*}
    For \(\calR\), we obtain the result by
        \begin{align*}
            \intO \abs{ \R(x, p(x)) \cdot \psi(x)} \dx
             & \leq \delta \intO \frac {\abs{q}\abs{\psi}}{\max(\abs{q},w)}\dleb
            \leq \delta \int_{\{q(x)\neq 0\}} \frac{\abs{q}}{\abs{q}}\abs{\psi}\dleb
            \\
             & \leq \delta\norm{\psi}_{\L1*} \leq \delta\abs{\Omega}^{\frac
                1r}\norm{\psi}_{\Lap*}<
            \infty\,,
        \end{align*}
        where \(1 = \frac 1{\alpha'} + \frac 1r\).
\end{proof}

Analogously to \(\G\), we will denote \(\calG := \calF + \calR\).

In light of the above differentiability results, we observe that~\cref{eq:ypde:reg} is nothing
    else than the necessary optimality conditions of
    \begin{equation}\label{eq:ymin}\tag{BP\(_\dagger\)}
    \min_{y \in \Wempty}
    \calJ(y) := \intO \calG(-\grad y)\dleb - \intO y\dmu.
\end{equation}

As \(\calF\) and \(\calR\) are convex,~\cref{eq:ypde:reg} is indeed sufficient for optimality so
    that~\cref{eq:ymin} is equivalent to~\cref{eq:ypde:reg}.
More precisely, \(\calG\) is uniformly convex for \(\delta>0\), as \(\tfrac{\alpha}{\alpha-1}\geq
    2\).
Now, we can peform a classical Armijo backtracking for \(\calJ\) as detailed in~\cref{alg:armijo}.
Note that
    \begin{equation*}
        \Deriv_y\calJ(y)\eta = -\intO \G(-\grad y) \cdot \grad \eta\dleb
        - \intO \eta\dmu
    \end{equation*}
    so that the Armijo condition in~\cref{alg:armijo} can be written as
    \begin{equation*}
    \begin{split}
    \intO \calG(-\grad y &- \sigma_k \grad\eta_k)\dleb
    > \intO \calG(-\grad y)\dleb\\
    &- \gamma\sigma_k\intO \G(-\grad y) \cdot \grad \eta\dleb
    +\sigma_k(1-\gamma)\intO\eta\dmu\,.
\end{split}
\end{equation*}

\begin{algorithm}
    \caption{Armijo line search for~\cref{eq:ymin}}\label{alg:armijo}
    \begin{algorithmic}
        \Require{} \(y, \eta\in\Wempty\),  \(\sigma_0 > 0\),
        \(\beta, \gamma \in \, (0, 1)\)
        \State{} \(k\gets 0\)
        \While{} {\(\calJ(y + \sigma_k \eta) > \calJ(y) + \gamma\,\sigma_k \Deriv_y\calJ(y)\eta\)}
        \State{} \(\sigma_{k+1} \gets \beta \sigma_k\)
        \State{} \(k \gets k+1\)
        \EndWhile{}
    \end{algorithmic}
\end{algorithm}

\subsection{Connection to Primal Problem}

We want to analyze the connection between problem~\cref{eq:reg_beckmann} and
    problem~\cref{eq:ymin}.

\begin{lemma}\label{thm:calG:conj}
    The Fenchel conjugate \(\calG*\) \wrt the second variable is given by
        \[
        \calG*(x,q) =
        \begin{cases}
        \abs{q}^2 \frac{w}{2\delta} - \delta w\,,&\abs{q}\leq \delta\,,\\
        \frac{\epsilon}{\alpha}{(\abs{q}-\delta)}^\alpha - \frac 32 \delta w +
        \abs{q}w\,,&\text{else}\,.
    \end{cases}
    \]
    Moreover, \(\calG\) is a normal integrand in the sense of~\cite[Definition
        14.27]{rockafellar:1998}.
\end{lemma}
\begin{proof}
    We begin by deriving the conjugate of \(\calG\).
    First note that
        \begin{equation*}
        \grad_p (s\cdot p - \calG(x,p)) = s - \G(x,p)\,.
    \end{equation*}
    We then observe
        \begin{equation*}
            {\G(x,\blank)}^{-1}(z) =
            \begin{cases}
                z \frac{w}\delta\,, & \abs{z}\leq\delta\,, \\
                \frac{z}{\abs{z}}\Big(\epsilon {(\abs{z} - \delta)}^{\alpha-1} +
                w\Big)\,,           & \text{else}\,,
            \end{cases}
        \end{equation*}
        so that we can insert \(p := {\G(x,\blank)}^{-1}(s)\) into \(s\cdot p - \calG(x,p)\).
    By
        straightforward manipulations, the first claim follows.

    For the second claim, we first note that \(\La*\) is decomposable relative to the Lebesgue
        measure in the sense of~\cite[Definition 14.59]{rockafellar:1998}.
    The assertion then follows by~\cite[Example 4.29]{rockafellar:1998} as \(\calG(x,\blank)\) is
        continuous for all \(x\in\Omega\) and \(\calG(\blank, p)\) is measurable for all
        \(p\in\RR^d\).
\end{proof}

Using the above result, we can characterize the connection as follows.

\begin{theorem}\label{thm:connection:primal}
    The predual problem to
        \begin{equation}\tag{BP\(_{\epsilon,\delta}\)}\label{eq:double_reg_beckmann}
            -\inf\left\{ \intO \calG*(q)\dleb \,\middle|\,	q\in\M\,,q\ll \leb^d\,, \dens q
            \in
            \La,
            \div \dens q = \mu \text{ in } \Wperp \right\}
        \end{equation}
        is given by~\cref{eq:ymin} and strong duality holds.
\end{theorem}
\begin{proof}
    Let
        \[
            f:\CC\to\RR\cup\{\pm\infty\}\,,\quad
            f(\xi) =
            \begin{cases}
                - \dual \xi\mu\,, & \xi\in \Wempty\,, \\
                \infty\,,         & \text{else}
            \end{cases}
        \]
        and \(g:\CC*\to \RR\cup\{\pm \infty\}\), \(g(\xi) = \intO\calG(\xi)\dleb\).
    It is easy to see that
        \[
        f^*(\nu) = \iota_{\{-\mu\}}(\nu) =
        \begin{cases}
        0\,,&\nu = -\mu \text{ in } \Wperp\,,\\
        \infty\,,&\text{else.
    }
    \end{cases}
    \]
    Finally, by~\cref{thm:calG:conj}, \(g^*(\nu) = \intO \calG*(q)\dleb \).
    The assertion then
        follows by standard arguments, see e.g.~\cite[Theorem 4.4.3]{borwein:2005}.
\end{proof}

\section{Approximation Results}\label{sec:approx}

Next we turn to results on approximation properties.
More precisely, we show that minimizers of
    the regularized problems converge to minimizers of~\cref{eq:beckmann} under suitable
    assumptions.

Recall from, e.g.,~\cite{braides:2002}, that a sequence \((F_n)\) of functionals \(F_n:X\to
    \RR\cup\{\infty\}\) on a metric space \(X\) is said to \(\Gamma\)-converge to a functional
    \(F:X\to\RR\cup\{\infty\}\), written \(F = \Gammalim_{n\to\infty} F_n\), if
    \begin{enumerate}
    \item for every sequence \(\{x_n\}\subset X\) with \(x_n\to x\), it holds
    \(%
    F(x) \leq \liminf_{n\to\infty} F_n(x_n)
    \) and%
    \item for every \(x\in X\), there is a sequence \(\{x_n\}\subset X\) with \(x_n\to x\) and
    \(%
    F(x) \geq \limsup_{n\to\infty} F_n(x_n)
    \).
This sequence is also called a \emph{recovery sequence}.
\end{enumerate}
It is a straightforward consequence of this definition that if \(F_n\) \(\Gamma\)-converges to
    \(F\) and \(x_n\) is a minimizer of \(F_n\) for every \(n\in \N\), then every cluster point of
    the
    sequence \((x_n)\) is a minimizer to \(F\).
Furthermore, \(\Gamma\)-convergence is stable under perturbations by continuous functionals.

To prove the desired approximation results, we will rely on smoothing of measures in order to
    construct the necessary recovery sequences.
Moreover, we need the following technical assumption.
\begin{assumption}\label{assmpt:starshaped}
    Assume that \(\Omega\) is strictly star shaped \wrt \(0\), \ie for all \(x\in\Omega\) and
        \(0\leq \lambda<1\), it holds \(\lambda x \in \Omega^\circ\).
\end{assumption}
\begin{remark}
    We leverage~\cref{assmpt:starshaped} in~\cref{thm:laplace_eps_continuous} below.
    However, while we only use a linear transformation of the domain in the following, the
        techniques we use in the proof of~\cref{thm:laplace_eps_continuous} could be applied in
        more general settings of nonlinear bi-Lipschitz deformations which would allow us to relax
        this assumption.
    We still focus on star shaped domains for the sake of brevity.
    Additionally,~\cref{assmpt:starshaped} is not overly restrictive as one can always
        formulate~\cref{eq:beckmann} on a strictly star shaped domain \(K\supset \Omega\) and
        approximate the original problem by choosing \(w\) to be large on \(K\setminus\Omega\).
\end{remark}
Throughout the rest of this section, for a given sequence \(0<\tau\to 0\), let
    \(0\leq\phi_\tau\in\test[\RR^d]\) be a sequence of mollifiers.
To avoid boundary effects, we will need to slightly extend the domain \(\Omega\).
More precisely, for every \(\tau>0\), we choose \(s>1\) such that \(\Omega_\tau := (1+s)\Omega
    \supset \Omega + \supp \phi_\tau\).
\Wloss we may assume \(\Omega_\tau \supset \Omega_\vartheta\) whenever \(\tau > \vartheta\).
Note that this is possible thanks to~\cref{assmpt:starshaped}.
Moreover, we denote \(\Tomega = \cup_\tau \Omega_\tau\).
Given a function (or measure) \(f\), we will denote by \(\tilde f\) the extension of \(f\) onto
    \(\Tomega\) by zero.
With \(\hat w\) we will denote a continuous extension of \(w\in \CC\) onto \(\Tomega\) which also
    satisfies \(\min_{\Tomega} \hat w  = \min_\Omega w\).

For \(\nu\in\MMperp\) and \(\nu_\epsilon\in\MMperp[\Emega]\) let now
    \(\H[\delta][{\nu_\epsilon}]:\M[\Emega] \to\RR\cup\{\pm\infty\}\) and \(\Hnull[\nu]:\M
    \to\RR\cup\{\pm\infty\}\) be defined by
    \[
        \H[\delta][\nu_\epsilon](q) =
        \begin{cases}
            \intO[\epsilon] \calG*(\dens q)\dleb\,, & q\ll \leb^d\,, \dens q \in \La[\Emega]*\,,
            \div
            \dens q = \nu_\epsilon \text{ in } \Wperp[\Emega]\,,
            \\
            \infty\,,                               & \text{else,}
        \end{cases}
    \]
    where \(\calG*(\blank, q)\) is extended onto \(\Emega\) by extending \(w\) with \(\hat w\),
    and
    \begin{equation*}
        \Hnull[\nu](q) =
        \begin{cases}
            \int_\Omega w \dabsq \,, & \mdiv q = \nu \text{ in } \MMperp\,, \\
            \infty\,,                & \text{else,}
        \end{cases}
    \end{equation*}
    respectively.
Note that
    \begin{equation*}
    \H[0][\nu_\epsilon](q) =
    \begin{cases}
    \int_{\Emega} \hat w\dabsq + \frac\epsilon\alpha \Lanorm[\Emega]{\dens q}^\alpha\,,&  q\ll
    \leb^d\,, \dens q \in \La[\Emega]*\,, \div \dens q = \nu_\epsilon \text{ in }
    \Wperp[\Emega]\,,\\
    \infty\,,&\text{else.
}
\end{cases}
\end{equation*}

Note that we can extend \(\H[\delta][\nu]\) and \(\Hnull[\nu]\) to be defined on measures on
    \(\Tomega\) by extending the argument onto \(\Tomega\) by zero as described above.
Strictly speaking, the approximating problems that we consider are given as problems on
    \(\Emega\),
    \ie
    \begin{equation*}
        \min_{q\in\M[\Emega]*} \H[0][\tilde\mu](q)
        \quad\text{and}\quad
        \min_{q\in\M[\Emega]*} \H[\delta][\tilde\mu](q)\,,
    \end{equation*}
    respectively.
For convenience, we will refer to these problems
    by~\cref{eq:reg_beckmann,eq:double_reg_beckmann}, too.

Before we present the first approximation result, we state two auxiliary results.
\begin{lemma}\label{thm:laplace_eps_continuous}
    Let~\cref{assmpt:general,assmpt:starshaped,assmpt:laplace} hold and
        let \(\tau >0\).
    Let \(\nu_\tau\in\Wperp[\Taumega]\).
    Then the elliptic equation
        \begin{equation*}
        \int_{\Taumega} \grad y_\tau \cdot \grad \psi \dleb = - \sca\psi {\nu_\tau} \qquad
        \forall\psi\in\Wempty[\Taumega] \,.
    \end{equation*}
    has a unique solution
    \(y_\tau\in \Wa<\emptyset>[\Taumega]\).
    Moreover, the solution operator
        \(\laplace_\tau\inv:\Wperp[\Taumega] \to  \Wa<\emptyset>[\Taumega]\) is uniformly bounded
        for
        \(\tau\to 0\).
\end{lemma}

A proof of~\cref{thm:laplace_eps_continuous} is given in~\cref{sec:appdx:laplace_eps_cont}.

\begin{lemma}\label{rem:extension_div}
    Let~\cref{assmpt:general,assmpt:laplace,assmpt:starshaped} hold and
        let \(q\in\M*\) and \(\mu\in\MMperp\).
    Let \(q_\epsilon\in\M[\Emega]*\) such that \(\tilde q_\epsilon \wsto \tilde q\) in
        \(\M[\Tomega]*\) and \(\mdiv q_\epsilon = \tilde \mu\) in \(\MMperp[\Emega]\).
    Then also \(\mdiv q = \mu\) in \(\MMperp\).
\end{lemma}
\begin{proof}
    Let \(\psi\in\CC[1]\), \(\eta>0\) and \wloss assume \(\epsilon < \eta\).
    Then with \(\psi_\eta := \psi(\blank \cdot {(1+\eta)}^{-1})\) it holds \(\psi_\eta|_{\Emega}
        \in
        \CC[1][\Emega]\) and \(\psi_\eta \to \psi\) in \(\CC[1]\).
    Let now \(\xi_\eta\in\CC[][\Tomega]*\) be a continuous extension of \(\grad \psi_\eta\) onto
        \(\Tomega\).
    Then
        \begin{equation*}
            -\intO \grad \psi_\eta \dmu
            =
            \intTO \xi_\eta \cdot \bd \tilde q_\epsilon \xrightarrow[\epsilon\to 0]{}
            \intTO \xi_\eta \cdot \bd \tilde q
            = \intO \grad \psi_\eta \cdot \bd q
        \end{equation*}
        and passing to the limit \(\eta\to 0\) concludes the proof.
\end{proof}

We are now in the position to state our first approximation result, which covers convergence of
    the minimizers of~\cref{eq:reg_beckmann}.

\begin{theorem}\label{thm:gammaconv:eps}
    Let~\cref{assmpt:general,assmpt:laplace,assmpt:starshaped} hold.
    It holds \(\Gammalim_{\epsilon\to 0} \H[0][\tilde \mu] = \Hnull\)  \wrt \ws convergence in
        \(\M[\Tomega]*\).
\end{theorem}
\begin{proof}
    \begin{enumerate}
        \item\emph{\(\liminf\)-condition:}
        Let \(q_\epsilon\in\M[\Emega]*\) be such that \(\tilde q_\epsilon \wsto q\in\M[\Tomega]\)
            in \(\M[\Tomega]\).
        Due to the \ws convergence, \((\tilde q_\epsilon)\) is bounded in \(\M[\Tomega]\), so that
            \(\intTO \hat w \dabsq<\infty\).
        Resort now to a subsequence such that \(\H[0][\tilde \mu](q_\epsilon) < \infty\).
        Then by~\cref{thm:conv:vector_measure:spt}, \(\spt q \subset \Omega\).
        Together with~\cref{rem:extension_div}, \(q\) is feasible for~\cref{eq:beckmann} and the
            assertion then follows directly from \(\intTO \hat w \bd\abs\blank\) being \lsc \wrt
            \ws convergence in \(\M[\Tomega]\) and \(\norm{\blank}_{\La[\Tomega]*} \geq 0\).

        \item\emph{\(\limsup\)-condition:}
        Let \(q\in\M*\) be arbitrary.
        In the case \(\Hnull(q) =\infty\), the assertion holds trivially.
        Hence, assume \(\Hnull(q) < \infty\).

        Let now \(\phi_\epsilon\) be as above and \wloss assume
            \(\epsilon \norm{\phi_\epsilon}_{\L\infty[\RR^d]}^\alpha \to 0\)
            for \(\epsilon \to 0\).
        Set \(q_\epsilon := \phi_\epsilon * \tilde q\) and \(\mu_\epsilon := \phi_\epsilon *
            \tilde \mu\).
        Then \(q_\epsilon\in\La[\Emega]*\) and \(\I(\tilde q_\epsilon)\to\tilde q\) in
            \(\M[\Tomega]*\) by~\cref{thm:conv:vector_measure}.
        Define now \(e_\epsilon := \tilde \mu_\epsilon - \tilde \mu\).
        It is straightforward to see that \(\sca{e_\epsilon}{1}=0\), \ie \(e_\epsilon\in
            \Wperp[\Emega]\).
        Then by~\cref{thm:laplace_eps_continuous} there is \(y_\epsilon \in
            \Wa<\emptyset>[\Emega]\)
            solving
            \begin{equation}\label{eq:y_epsilon}
            \int_{\Emega} \grad y_\epsilon \cdot \grad \psi \dleb = - \int_{\Emega} \psi \bd
            e_\epsilon\qquad \forall\psi\in\Wempty[\Emega] \,.
        \end{equation}
        Moreover,~\cref{remark:embedding,thm:conv:vector_measure} yield \(e_\epsilon \to 0\) in
            \(\Wperp[\Emega]\) and hence, \(\tilde y_\epsilon \to 0\) in \(\Wempty[\Tomega]\)
            by~\cref{thm:laplace_eps_continuous}.
        Thus, \(\widetilde {\grad y_\epsilon} \to 0\) in \(\La[\Tomega]\)
            and by defining \(\check q_\epsilon\in\M[\Emega]\) as
            \begin{equation*}
                \check q_\epsilon := q_\epsilon + {\grad y_\epsilon},
            \end{equation*}
            we obtain \(\widetilde{\I(\check q_\epsilon)} \to \tilde q\) in \(\M[\Tomega]*\).
        For \(\psi\in\Wempty[\Emega]\),
            leveraging~\cref{thm:conv:vector_measure}~\labelcref{thm:conv:vector_measure:div} now
            yields
            \begin{align*}
                -\int_{\Emega} \grad \psi \cdot \bd \check q_\epsilon
                 & = -\int_{\Emega} \grad\psi \cdot q_\epsilon \dleb - \int_{\Emega} \grad \psi
                \cdot \grad
                y_\epsilon \dleb
                \\
                 & = \int_{\Emega} \psi \bd \mu_\epsilon + \intO \psi \dmu - \int_{\Emega} \psi
                \bd
                \mu_\epsilon
                = \int_{\Emega} \psi \bd \tilde\mu\,,
            \end{align*}
            so that \(\div \check q_\epsilon = \tilde \mu\) in \(\Wperp[\Emega]\).
        Thus, \(\check q_\epsilon\) is feasible
            for~\cref{eq:reg_beckmann}.
        Going on, we note that \(\intO w \abs{\check q_\epsilon} \dleb \to \intO w\dabsq\) due to
            \(\widetilde{\I(\check q_\epsilon)} \to \tilde q\in\M[\Tomega]\).
        Moreover,
            \begin{equation}\label{eq:convq_upper:1}
                \abs{q_\epsilon(x)} \leq \sup_{h\in\RR^d\,,\,\abs{h}\leq 1} \abs[\Big]{\intTO
                    \phi_\epsilon(x-y)\bd (\tilde q(y)\cdot h)} \leq \intTO
                \abs[\big]{\phi_\epsilon(x-y)} \bd
                \abs{\tilde q(y)}\,,
            \end{equation}
            which gives
            \begin{equation}\label{eq:convq_upper:2}
            \norm{q_\epsilon}_{\L\infty[\Emega]} \leq	  \norm{\phi_\epsilon}_{\L\infty[\RR^d]}
            \abs{q}(\Omega)\,.
        \end{equation}
        Hence,
            \[
                \of{\frac \epsilon\alpha}^{\frac 1\alpha} \norm{\check q_\epsilon}_{\La[\Emega]*}
                \leq \of{\frac \epsilon\alpha}^{\frac 1\alpha}
                \left(
                \abs{\Tomega} \norm{\phi_\epsilon}_{\L\infty[\RR^d]} (\abs{q}(\Omega))
                + \norm{\grad y_\epsilon}_{\La[\Emega]*}
                \right)
                \,,
            \]
            which, due to the assumption on \(\phi_\epsilon\), vanishes for \(\epsilon\to 0\).
        This yields
            the desired assertion and concludes the proof.
        \qedhere
    \end{enumerate}
\end{proof}

\begin{corollary}\label{thm:gammaconv:eps:corollary}
    In the setting of~\cref{thm:gammaconv:eps},
        let \(w\geq \wlower > 0\).
    Let \(\epsilon_n>0\) be a vanishing sequence and \((q_n)\subset\La[\Omega_{\epsilon_n}]*\) be
        the sequence of corresponding solutions of~\cref{eq:reg_beckmann}.
    Then \((q_n)\) admits a subsequence that converges to a solution of~\cref{eq:beckmann} \wrt
        \ws convergence in \(\M[\Tomega]\).
\end{corollary}
\begin{proof}
    Let \(q_0 \in \La*\) be fixed such that \(\div q_0 = \mu\) in \(\Wperp\), which exists due
        to~\cref{thm:div:surjective}.
    Then \(q_n\) satisfies
        \begin{equation*}
        \intTO \hat w \abs{\tilde q_n} \dleb + \frac{\epsilon_n}{\alpha}
        \Lanorm[\Tomega]{\tilde q_n}^\alpha
        \leq \intTO \hat w \abs{\tilde q_0} \dleb + \frac{\epsilon_n}{\alpha}
        \Lanorm[\Tomega]{\tilde q_0}^\alpha \,.
    \end{equation*}
    Thus, due to \(\hat w \geq \wlower > 0\) and \(\frac{\epsilon_n}{\alpha}
        \Lanorm[\Tomega]{\tilde{q}_0}^\alpha\to 0\) for \(n\to\infty\), it holds
        \begin{equation*}
        \wlower \norm{\tilde q_n}_{\L1[\Tomega]} \leq \intTO \hat w \abs{\tilde q_0} \dleb <
        \infty\,.
    \end{equation*}
    Hence, \((\I(\tilde q_n))\) is bounded in \(\M[\Tomega]\) and by the Banach-Alaoglu theorem
        there exists a subsequence (denoted by the same symbol), which converges to \(q\in
        \M[\Tomega]*\)
        \wrt \ws convergence in \(\M[\Tomega]\).
    The assertion then follows directly from~\cref{thm:gammaconv:eps} and the properties of
        \(\Gamma\)-convergence.
\end{proof}

\subsection{Convergence for vanishing Huber regularization}

Going on, we turn to problem~\cref{eq:double_reg_beckmann}.
As a first step, we only consider the
    convergence for \(\delta\to 0\).
We start by proving an auxiliary result.

\begin{lemma}\label{thm:Hdelta:locally_uniform}
    Let~\cref{assmpt:general} hold.
    For \(\delta\to 0\), the functional \(q\mapsto \intO[\epsilon] \calG*(q)\dleb\) converges
        locally uniformly to \(q\mapsto \intO[\epsilon] \calG[\epsilon,0]*(q)\dleb\) on
        \(\La[\Emega]*\).
\end{lemma}
\begin{proof}
    Let \(K\subset \La[\Emega]*\) be a bounded set.
    We want to show
        \[
        \lim_{\delta\to 0}\sup_{q\in K} \abs[\Big]{ \intO[\epsilon] \calG*(q) -
            \calG[\epsilon,0]*(q) \dleb } = 0\,.
    \]
    Clearly,
    \begin{align*}
        \abs[\Big]{ \intO[\epsilon] \calG*(q) - \calG[\epsilon,0]*(q) \dleb }
         & \leq \abs[\Big]{\int_{\{\abs{q(x)}\leq \delta\}} \calG*(x,q(x))  - \hat w \abs{q(x)} -
            \frac\epsilon\alpha \abs{q(x)}^\alpha \dx}
        \\
         & \quad + \abs[\Big]{ \int_{\{\abs{q(x)}> \delta\}} \calG*(x,q(x))  - \hat w \abs{q(x)}
            -
            \frac\epsilon\alpha \abs{q(x)}^\alpha \dx }\,.
    \end{align*}
    We first consider the first term:
        \begin{align*}
            \abs[\Big]{\int_{\{\abs{q}\leq \delta\}} \calG*(x,q(x))  - \hat w \abs{q(x)} -
                \frac\epsilon\alpha \abs{q(x)}^\alpha \dx}
             & =\abs[\Big]{\int_{\{\abs{q}\leq \delta\}} \abs{q}^2 \frac{\hat w}{\delta} - \delta
                \hat w
                - \hat w \abs{q} - \frac\epsilon\alpha \abs{q}^\alpha \dleb}
            \\
             & \leq 3\delta \norm{\hat w}_{\L1[\Emega]} + \frac \epsilon\alpha \delta^\alpha
            \abs{\Emega} \xrightarrow[\delta\to 0]{} 0\,,
        \end{align*}
        independent of \(q\).
    For the second term it holds
        \begin{align*}
        \abs[\Big]{ \int_{\{\abs{q}> \delta\}} \calG*(x,q(x)) - \hat w \abs{q(x)} -
            \frac\epsilon\alpha {\abs{q(x)}}^\alpha \dx }
        & =\abs[\Big]{\int_{\{\abs{q}> \delta\}} \frac\epsilon\alpha
        \of[\big]{{(\abs{q} - \delta)}^\alpha	- {\abs{q}}^\alpha}
        - \frac 32\delta \hat w \dleb }
        \\
        & \leq \abs[\Big]{\int_{\{\abs{q}> \delta\}} \frac\epsilon\alpha
        \of[\big]{{(\abs{q} -  \delta)}^\alpha - {\abs{q}}^\alpha }\dleb} +
        \frac 32 \delta \norm{\hat w}_{\L1[\Emega]}\,.
    \end{align*}
    Denoting \(Q_\delta: \RR\to\RR\), \(Q_\delta(x) = \max\{x,\delta\}\), we see that
        \begin{align*}
        0 &> \int_{\{\abs{q}> \delta\}} {(\abs{q} - \delta)}^\alpha  - {\abs{q}}^\alpha\dleb
        = \int_{\{\abs{q}> \delta\}} {(Q_\delta(\abs{q}) - \delta)}^\alpha  -
        {Q_\delta(\abs{q})}^\alpha\dleb\\
        &> \intO[\epsilon] {(Q_\delta(\abs{q}) - \delta)}^\alpha	-
        {Q_\delta(\abs{q})}^\alpha\dleb\,.
    \end{align*}
    By noting that \(Q_\delta(\abs{q}) - \delta\) and \(Q_\delta(\abs{q})\) are non-negative, we
        thus obtain
        \begin{align*}
        \sup_{q\in K} \frac\epsilon\alpha \abs[\Big]{\int_{\{\abs{q}> \delta\}}  {(\abs{q} -
                \delta)}^\alpha - \abs{q}^\alpha\dleb}
        \leq & \sup_{q\in K} \frac\epsilon\alpha \abs[\big]{ \norm{Q_\delta(\abs{q}) -
                \delta}_{\La[\Emega]}^\alpha - \norm{Q_\delta(\abs{q})}^\alpha_{\La[\Emega]} }\,.
    \end{align*}
    Because \({(\blank)}^\alpha\) is locally Lipschitz on \(\RR\), there is a constant \(C_K\geq
        0\)
        such that
        \(\abs{\Lanorm[\Emega]{v}^\alpha - \Lanorm[\Emega]{w}^\alpha} \leq C_K
        \abs{\Lanorm[\Emega]{v}
            - \Lanorm[\Emega]{w}}\)
        for all \(v,w\in K\).
    Together with the reverse triangle inequality, this yields
        \begin{align*}
            \sup_{q\in K} \frac\epsilon\alpha \abs[\big]{ \norm{Q_\delta(\abs{q}) -
                    \delta}_{\La[\Emega]}^\alpha - \norm{Q_\delta(\abs{q})}^\alpha_{\La[\Emega]} }
            \leq & \sup_{q\in K} \frac\epsilon\alpha C_K \abs[\big]{ \norm{Q_\delta(\abs{q}) -
            \delta}_{\La[\Emega]} - \norm{Q_\delta(\abs{q})}_{\La[\Emega]} }                   \\
            \leq & \sup_{q\in K} \frac\epsilon\alpha C_K \norm{\delta}_{\La[\Emega]} =
            \frac\epsilon\alpha C_K \abs{\Emega} \delta \xrightarrow[\delta\to 0]{} 0
        \end{align*}
        and concludes the proof.
\end{proof}

Now we're in a position to prove the desired result on \(\Gamma\)-convergence.

\begin{theorem}\label{thm:gammaconv:delta}
    Let~\cref{assmpt:general,assmpt:laplace,assmpt:starshaped} hold.
    Then for \(\nu\in\MMperp[\Emega]\) it holds \(\Gammalim_{\delta\to 0}\H[\delta][\nu] =
        \H[0][\nu]\)  \wrt \ws convergence in \(\M[\Emega]*\).
\end{theorem}
\begin{proof}
    \begin{enumerate}
        \item\emph{\(\limsup\)-condition:}
        Let \(q\in\M[\Emega]*\) be arbitrary.
        As recovery sequence,
            we use the constant sequence, \ie \(q_\delta \equiv q\).
        In the case \(\H[0][\nu](q) = \infty\), the assertion holds trivially.

        Hence, we assume \(\H[0][\nu](q) <\infty\).
        In this case, \(\dens q - \delta\in \La[\Emega]*\) with \(\div \dens q = \nu\) in
            \(\Wperp[\Emega]\) and thus, \(\H[\delta][\nu](q)<\infty\).
        Then, by~\cref{thm:Hdelta:locally_uniform}, \(\H[\delta][\nu](q) \to \H[0][\nu](q)\) for
            \(\delta\to 0\).

        \item\emph{\(\liminf\)-condition:}
        Let \(q\in\M[\Emega]*\) be arbitrary and let \(\M[\Emega]*\ni q_\delta \wsto q\) in
            \(\M[\Emega]*\).
        Moreover, we have \(\int_{\Emega} \hat w \dabsq <\infty\) analogously to the proof
            of~\cref{thm:gammaconv:eps}

            First, assume \(\H[0][\nu](q)< \infty\) and \wloss resort to a subsequence of
            \(q_\delta\)
            (denoted by the same symbol) such that
            \(\lim_{\delta\to 0} \H[\delta][\nu](q_\delta) = \liminf_{\delta\to 0}
            \H[\delta][\nu](q_\delta)<\infty\).
        We may \wloss assume \(\delta\leq 1\) so that
            \begin{equation}\label{eq:gammacondelta:lower_est}
            \calG*(x,p)
            \geq \hat w\of{\abs{q} - \frac 32} + \frac\epsilon\alpha {\of{{(\abs{p} -
                            1)}_+}}^\alpha\,.
        \end{equation}
        Thanks to~\cref{thm:div:surjective} we may choose \(q_0\in\La[\Emega]*\) fixed with \(\div
            q_0
            = \nu\) in \(\Wperp[\Emega]\) and obtain
            \begin{equation}\label{eq:gammaconvdelta:lower}
                \int_{\Emega}\frac \epsilon\alpha {({(\abs{\dens{q_\delta}} -1)}_+)}^\alpha \leq
                \intO[\epsilon]\calG*(\dens{q_\delta}) \dleb
                \leq \intO[\epsilon]\calG*(q_0) \dleb
                < \infty
            \end{equation}
            and hence \(\dens {q_\delta}\) is bounded in \(\La[\Emega]*\), \ie there is some
            \(K\subset\La[\Emega]*\) bounded such that \((\dens{q_\delta})\subset K\).
        This also yields weak convergence of a subsequence of \(\dens{q_\delta}\) in
            \(\La[\Emega]*\)
            and together with \(q_\delta\wsto q\) in \(\M[\Emega]*\) we have
            \(\dens{q_\delta}\wto\dens q\) in
            \(\La[\Emega]*\).
        Hence,~\cref{thm:Hdelta:locally_uniform} yields
            \begin{align*}
                \liminf_{\delta\to 0} \H[\delta][\nu](q_\delta)
                 & \geq \liminf_{\delta\to 0} (\H[\delta][\nu](q_\delta) - \H[0][\nu](q_\delta)) +
                \liminf_{\delta\to 0} \H[0][\nu](q_\delta)
                \\
                 & \geq - \sup_{p\in K}\abs[\big]{\H[\delta][\nu](\I(p)) \H[0][\nu](\I(p))} +
                \liminf_{\delta\to 0} \H[0][\nu](q_\delta)
                \\
                 & = \liminf_{\delta\to 0} \H[0][\nu](q_\delta) \geq \H[0][\nu](q)\,,
            \end{align*}
            where the last inequality holds due to \(\H[0][\nu]\) being \lsc \wrt weak convergence
            in \(\La[\Emega]*\).

        Assume now that  \(\H[0][\nu](q) = \infty\) with either \(\frac{\dq}{\dleb}\not\in
            \La[\Emega]*\) or \(q\not\ll\leb^d\).
        For a contradiction, assume \(\liminf_{\delta\to
                0}\H[\delta][\nu](q_\delta) < \infty\).
        As seen in~\cref{eq:gammaconvdelta:lower},
            this implies boundedness of \(\dens {q_\delta}\) in
            \(\La[\Emega]*\) and as above, we obtain \(\dens{q_\delta}\wto\dens q\) in
            \(\La[\Emega]*\), which
            is the desired contradiction.

        Finally, we are left with the case \(\H[0][\nu](q) = \infty\) with \(q\ll\leb^d\) and
            \(\dens{q}\in \La[\Emega]*\) but \(\div \dens q \neq \nu\) in \(\Wperp[\Emega]\).
        For a contradiction, we assume
            \(\liminf_{\delta\to 0}\H[\delta][\nu](q_\delta) < \infty\)
            and pass to a subsequence (denoted by the same symbol) such that
            \(\H[\delta][\nu](q_\delta)\)
            converges.
        As above, it follows that \(\dens{q_\delta}\wto \dens{q}\) in \(\La[\Emega]*\) and
            therefore
            \begin{equation*}
            -\dual{\phi}{\tilde\nu} = \intO[\epsilon] \dens{q_\delta} \cdot \grad\phi\dleb \to
            \intO[\epsilon] \dens{q} \cdot \grad\phi\dleb
            \qquad\forall\,\phi\in \Wempty[\Emega]\,.
        \end{equation*}
        This implies
            \(\div \dens q = \nu\) in \(\Wperp[\Emega]\), thus yielding the desired contradiction
            and
            concluding the proof.
        \qedhere
    \end{enumerate}
\end{proof}

\begin{corollary}
    In the setting of~\cref{thm:gammaconv:delta},
        let \(\delta_n>0\) be a vanishing sequence and \((q_n)\subset\La[\Emega]*\) be the
        sequence of
        corresponding solutions of~\cref{eq:double_reg_beckmann}.
    Then \((q_n)\) admits a subsequence that converges to a solution of~\cref{eq:reg_beckmann}
        \wrt \ws convergence in \(\M[\Emega]\).
\end{corollary}
\begin{proof}
    Analogously to the argument involving~\cref{eq:gammaconvdelta:lower} in the proof
        of~\cref{thm:gammaconv:delta}, we obtain a subsequence of \(q_n\) (denoted by the same
        symbol) with
        \(q_n\wto q\in\La[\Emega]*\) in \(\La[\Emega]*\).
    This also implies \(\I(q_n)\wsto \I(q)\) in \(\M[\Emega]*\) and the assertion follows directly
        from~\cref{thm:gammaconv:delta} and the properties of \(\Gamma\)-convergence.
\end{proof}

\subsection{Simultaneous Convergence of \texorpdfstring{\(\epsilon\)}
    {\unichar{"03B5} 
    } and \texorpdfstring{\(\delta\)}
    {\unichar{"03B4} 
    }
}

Lastly, we aim to show \(\Gamma\)-convergence for \(\delta\to 0\) and \(\epsilon\to 0\)
    \emph{simultaneously}.

\begin{theorem}\label{thm:gammaconv:eps_delta}
    Let~\cref{assmpt:general,assmpt:laplace,assmpt:starshaped} hold and
        let \(h:\RR_+ \to \RR_+\) such that \(h(\delta)\to\infty\) for \(\delta\to 0\).
    Denote \(\tau := (\epsilon, \delta)\) and let \(\tau\to 0\) such that
        \begin{equation}\label{eq:tau:conv}
        \epsilon \cdot {h(\delta)}^\alpha\to 0.
    \end{equation}
    Then \(\Gammalim_{\tau\to 0}\H[\delta][\tilde \mu] = \Hnull\) \wrt \ws convergence in
        \(\M[\Tomega]*\).
\end{theorem}
\begin{proof}
    In the following, we will abbreviate \(\Htau := \H[\delta][\tilde \mu]\).

    \begin{enumerate}
        \item\emph{\(\liminf\)-condition:}
        Let \(q\in\M[\Tomega]*\) and
            \((q_\tau)\subset\M[\Omega_\tau]*\) such that \(\tilde q_\tau\wsto q\) in
            \(\M[\Tomega]*\).
        Analogously to the \(\liminf\) case in the proof of~\cref{thm:gammaconv:eps}, we obtain
            that
            \(q\) is feasible for~\cref{eq:beckmann} with \(\intO w \dabsq < \infty\).

        Without renaming, we resort to a subsequence of \((q_\tau)\) such that \(\lim_{\tau\to 0}
            \Htau(q_\tau) = \liminf_{\tau\to 0}\Htau(q_\tau)<\infty\).
        Then we obtain
            \begin{align*}
                \liminf_{\tau\to 0} \int_{\big\{\abs[\big]{\dens{\tilde q_\tau}}>\delta\big\}}
                \hat w\bd\abs{\tilde q_\tau(x)}
                 & = \liminf_{\tau\to 0} \intTO \hat w\bd\abs{\tilde q_\tau(x)} -
                \int_{\big\{\abs[\big]{\dens{\tilde q_\tau}}\leq \delta\big\}}
                \hat w\bd\abs{\tilde q_\tau(x)}
                \\
                 & \geq \liminf_{\tau\to 0} \intTO \hat w\bd\abs{\tilde q_\tau(x)} - \delta
                \norm{\hat w}_{\L1[\Tomega]}
                = \intO w \bd\abs{q}
            \end{align*}
            thanks to \(\intTO \hat w \bd\abs\blank\) being \lsc \wrt \ws convergence in
            \(\M[\Tomega]\).
        Going on, we see that
            \begin{equation*}
                \int_{\big\{\abs[\big]{\dens{\tilde q_\tau}} > \delta\big\}} \frac \epsilon\alpha
                {(\abs{\dens{\tilde q_\tau}} -\delta)}^\alpha - \frac 32 \delta \hat w \dleb
                \geq - \frac 32 \delta \norm{\hat w}_{\L1[\Tomega]} \xrightarrow[\tau\to 0]{} 0\
            \end{equation*}
            as well as
            \begin{equation*}
            \int_{\big\{\abs[\big]{\dens{\tilde q_\tau}} \leq \delta\big\}}
            \abs{\dens{\tilde q_\tau}}^2 \frac {\hat w}{2\delta}
            - \delta \hat w \dleb\geq \delta \norm{\hat w}_{\L1[\Tomega]}
            \xrightarrow[\tau\to 0]{} 0\,.
        \end{equation*}
        In summary, this yields \(\liminf_{\tau\to 0} \Htau(q_\tau) \geq  \Hnull(q)\).

        \item\emph{\(\limsup\)-condition:}
        Let \(q\in\M*\) be arbitrary.
        In the case \(\Hnull(q)=\infty\) the assertion again holds trivially.

        Hence, let \(\Hnull(q)<\infty\).
        Let \((\phi_\tau)\) be as above and \wloss assume \(\norm{\phi_\tau}_{\L\infty[\RR^d]}\leq
            h(\delta)\).
        Denote
            \(q_\tau := \phi_\tau*q\in\La[\Omega_\tau]*\)
            and similarly for \(\mu_\tau\).
        By~\cref{thm:conv:vector_measure}, it holds
            \(\I(\tilde q_\tau)\to \tilde q\) in \(\M[\Tomega]*\) and \(\mdiv \I(q_\tau) =
            \mu_\tau\) in
            \(\MMperp[\Tomega]\).
        Moreover, we obtain \(\norm{\tilde q_\tau}_{\L\infty[\Tomega]} \leq C h(\delta)\) for some
            constant \(C>0\) similarly to~\cref{eq:convq_upper:1,eq:convq_upper:2}.
        Analogously to the proof of~\cref{thm:gammaconv:eps}, we set
            \(\check q_\tau := \tilde q_\tau + {\grad y_\tau}\),
            where \(y_\tau\in\Wempty[\Omega_\tau]\) solves
            the analogue to~\cref{eq:y_epsilon}.
        Then we have
            \(\widetilde{\I(\check q_\tau)}\to \tilde q\) in \(\M[\Tomega]*\),
            \(\div \check q_\tau = \tilde \mu\) in \(\Wperp[\Tomega]\)
            and \(\widetilde{\grad y_\tau} \to 0\) in \(\La[\Tomega]*\).
        Going on, it holds
            \begin{equation*}
                \int_{\{\abs{\check q_\tau}<\delta\}} \abs{\check q_\tau}^2\frac {\hat w}{2\delta}
                -\delta \hat w \dleb\leq \intTO \delta \hat w - \delta \hat w\dleb = 0
            \end{equation*}
            and
            \begin{equation}\label{eq:gammaconv:eps_delta:objective}
            \int_{\{\abs{\check q_\tau}\geq \delta\}}\calG*(\check q_\tau)\dleb \leq
            \frac\epsilon\alpha\of{\norm{\check q_\tau + \widetilde{\grad
                        y_\tau}}_{\La[\Tomega]*}}^\alpha +
            \intTO \abs{\check q_\tau}\hat w\dleb\,.
        \end{equation}
        Due to~\cref{eq:tau:conv},
            \begin{equation*}
            \of{\frac\epsilon\alpha}^{\frac 1\alpha} \norm{\tilde q_\tau}_{\La[\Tomega]*}
            \leq C \of{\frac\epsilon\alpha}^{\frac 1\alpha} h(\delta) \xrightarrow[\tau \to 0]{}
            0\,.
        \end{equation*}
        Together with the strong convergence of
            \(\I(\check q_\tau)\) for \(\tau\to 0\) and
            \(\widetilde{\grad y_\tau} \to 0\) in \(\La[\Tomega]*\),
            the right hand side of~\cref{eq:gammaconv:eps_delta:objective} converges to
            \(\Hnull(q)\),
            thus concluding the proof.
        \qedhere
    \end{enumerate}
\end{proof}

\begin{corollary}
    In the setting of~\cref{thm:gammaconv:eps_delta},
        let \(w\geq \wlower > 0\) and
        let \(\epsilon_n, \delta_n>0\) be a vanishing sequences such that~\cref{eq:tau:conv}
        holds.
    Let \((q_n)\subset\La[\Omega_{\epsilon_n}]*\) be the sequence of corresponding solutions
        of~\cref{eq:double_reg_beckmann}.
    Then \((q_n)\) admits a subsequence that converges to a solution of~\cref{eq:beckmann} \wrt
        \ws convergence in \(\M[\Emega]\).
\end{corollary}
\begin{proof}
    Let \(q_0 \in \La*\) be fixed such that \(\div q_0 = \mu\) in \(\Wperp\), which exists due
        to~\cref{thm:div:surjective}.
    Then thanks to~\cref{eq:gammacondelta:lower_est} \(q_n\) satisfies
        \begin{align*}
            \intTO \hat w \of{\abs{\tilde q_n} - \frac 32} + \frac {\epsilon_n}\alpha
            {(({(\abs{\tilde q_n} -1)}_+))}^\alpha \dleb
            \leq \intTO \calG[\epsilon_n,\delta_n]*(\tilde q_n) \dleb
            \leq \intTO \calG[\epsilon_n,\delta_n]*(\tilde q) \dleb\,,
        \end{align*}
        where for the right hand side it holds
        \begin{equation*}
            \limsup_{n\to\infty }\intTO \calG[\epsilon_n,\delta_n]* (\tilde q_0) \dleb
            \leq \intTO \hat w \tilde q_0 \dleb < \infty
        \end{equation*}
        by similar argumentation as in \(\limsup\) case of the preceding proof.
    Thus, due to \(w\geq\wlower>0\), the sequence \((\tilde{q}_n)\) is bounded in
        \(\L1[\Tomega]*\)
        and the assertion then follows by argumentation analogous to the
        proof~\cref{thm:gammaconv:eps:corollary}.
\end{proof}

\section{Numerical Examples}\label{sec:numerics}

In this section, we report on the conducted numerical experiments.
We start by briefly explaining
    our discretization scheme.

\subsection{Discretization via Finite Elements}

To discretize the Newton equation in~\cref{eq:ssn:step:reg}, we employ standard piecewise linear
    and continuous finite elements.
The nodal basis associated with nodes \(x_1, \ldots, x_n\), \(n\in\N\), of a given triangular grid
    is denoted by
    \(\phi_1, \ldots, \phi_n\) such that the discretized ansatz and trial space is
    \(V_n = \operatorname{span}(\phi_1, \ldots, \phi_n)\).
Now, given an iterate \(y_h \in V_n\), the discrete counterpart of~\cref{eq:ssn:step:reg} reads
    \begin{equation}\label{eq:ssn:step:disc}
    \begin{aligned}
    \eta_h \in V_n, \quad \intO \eta_h \dleb &= 0,\\
    \intO \G'(-\grad y_h) \grad \eta_h \cdot \grad \psi\dleb
    &= \intO \G(-\grad y_h) \cdot \grad \psi\dleb + \intO \psi\dmu\\
    & \quad \forall\psi\in V_n:\,\intO\psi \dleb = 0\,.
\end{aligned}
\end{equation}
We introduce the matrices
    \begin{equation*}
        {A(y_h)}_{ij} := \intO \G'(-\grad y_h) \grad \phi_i \cdot \grad \phi_j\dleb,\quad
        M_{ij} := \intO \phi_i\phi_j\dleb
    \end{equation*}
    and the vectors
    \begin{equation}\label{eq:disc:db}
    {b(y_h)}_i := \intO \G(-\grad y_h) \cdot \grad \phi_i\dleb\,, \quad
    d_i := \intO \phi_i\dmu.
\end{equation}
Then~\cref{eq:ssn:step:disc} is equivalent to
    \begin{align}
        \one^\top M \eta    & = 0, \nonumber
        \\
        \eta^{\top} A(y_h)v & = {(b(y_h) + d)}^{\top}v\quad\forall v\in\RR^n:\,\one^{\top} Mv = 0
        \label{eq:meanzero}
    \end{align}
    where \(\eta\in \RR^n\) denotes the coefficient vector of \(\eta_h\) for the basis
    \(\phi_1,\ldots,\phi_n\) and \(\one = {[1, \ldots, 1]}^\top\).
If we introduce a scalar Lagrange multiplier \(r\) associated with~\cref{eq:meanzero}, then the
    system is equivalent to the
    saddle point problem
    \begin{equation*}
        \begin{pmatrix}
            A(y_h)      & M\one \\
            \one^\top M & 0
        \end{pmatrix}
        \begin{pmatrix}
            \eta \\ r
        \end{pmatrix}
        =
        \begin{pmatrix}
            b(y_h) + d \\
            0
        \end{pmatrix}
    \end{equation*}

    \begin{remark}
    If \(w\) is chosen piecewise constant on the triangular grid,
    the entries of \(A(y_h)\) and \(b(y_h)\) can be evaluated exactly, since \(\grad y_h\) is
    constant on each element.
The same holds for the objective \(\calJ(y_h)\) in the Armijo line search, as the second
    integral only involves linear combinations of piecewise linear functions on the elements,
    which can be integrated exactly.
\end{remark}

\subsection{\texorpdfstring{Influence of \(\epsilon\) and \(\delta\)}{{Influence of
                \unichar{"03B5} and \unichar{"03B4}}} 
}

We first illustrate effect of the regularization parameters \(\epsilon\) and \(\delta\) on the
    solutions of~\cref{eq:double_reg_beckmann}.
We choose a simple Friedrich-Keller grid, \ie we divide the domain \(\Omega = {[0,1]}^2\) into a
    regular partition of equally sized squares and divide each square into two congruent
    triangles.

Both the marginals \(\mu^+\), \(\mu^-\) and the cost function \(w\) are non-negative functions
    which are constant on the squares, \ie they are constant across two adjacent triangles.
For the exponent \(\alpha\), we choose \(\alpha=2\).
Note that we required \(\alpha < \tfrac
    {d}{d-1}\) in~\cref{assmpt:general}, so that \(\alpha=2\) is actually a limit case.
We start the iteration with \(y \equiv 0\) and use the parameters \(\sigma_0=1\), \(\beta=\tfrac
    12\), \(\gamma=\tfrac 1{10}\) for the Armijo line search (\cref{alg:armijo}).
As stopping criterion, we use the relative error of the optimality condition~\cref{eq:ypde:reg}.
More precisely, for each \(i=1,\ldots,n\) we calculated \(d_i\) and \({b(y)}_i\) as
    in~\cref{eq:disc:db}
    and use the relative error
    \begin{equation}\label{eq:stopping_crit}
        \frac{\abs{d-b(y)}}{\abs{b(y)}}
    \end{equation}
    and stopped the iteration once this error dropped below \(10^{-8}\) or after 1000 iterations.

\begin{figure}[t]
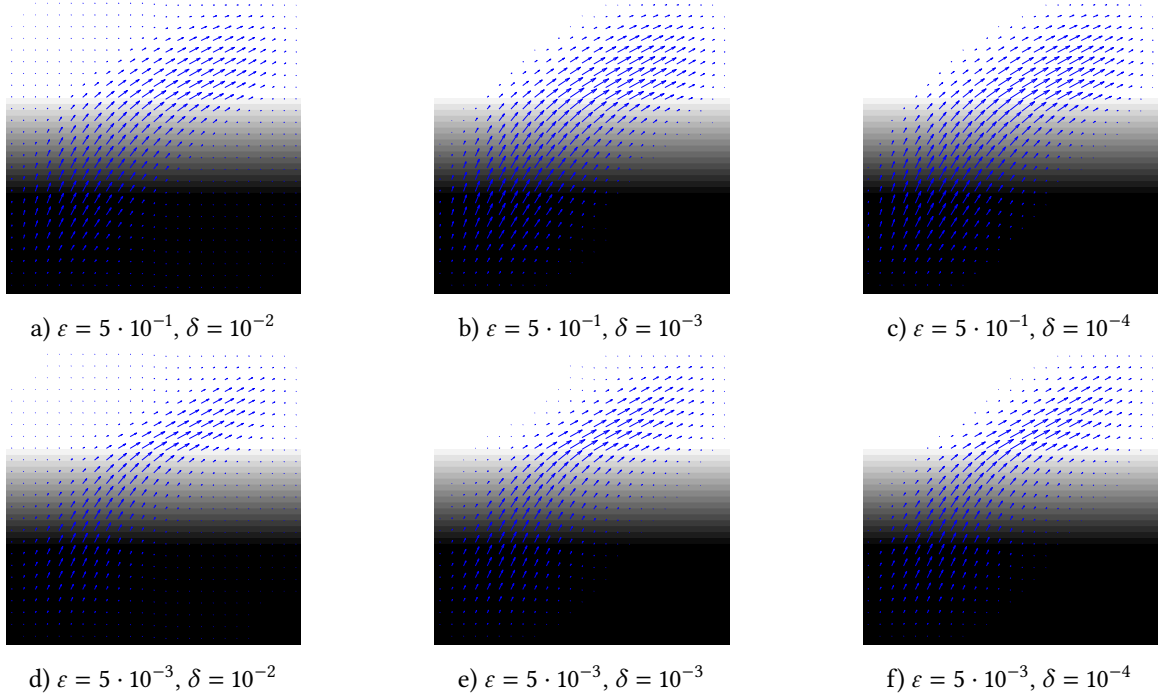

    \begin{center}
        \includeplot{toy/field_eps_5e-01_delta_1e-02.pdf}{\(\epsilon=5\cdot 10^{-1}\),
        \(\delta=10^{-2}\)}
        \includeplot{toy/field_eps_5e-01_delta_1e-03.pdf}{\(\epsilon=5\cdot 10^{-1}\),
        \(\delta=10^{-3}\)}
        \includeplot[0pt]{toy/field_eps_5e-01_delta_1e-04.pdf}{\(\epsilon=5\cdot 10^{-1}\),
        \(\delta=10^{-4}\)}\\
        \includeplot[\defaultfiguresep]{toy/field_eps_5e-03_delta_1e-02.pdf}{\(\epsilon=5\cdot
        10^{-3}\), \(\delta=10^{-2}\)}
        \includeplot{toy/field_eps_5e-03_delta_1e-03.pdf}{\(\epsilon=5\cdot 10^{-3}\),
        \(\delta=10^{-3}\)}
        \includeplot[0pt]{toy/field_eps_5e-03_delta_1e-04.pdf}{\(\epsilon=5\cdot 10^{-3}\),
        \(\delta=10^{-4}\)}
        \caption{Visualization of the flow field for an example with piecewise linear/constant
            cost
            function.
            Both \(\mu^+\) and \(\mu^-\) are Gaussians centered in the bottom left and top right
                quadrant, respectively.
        }\label{fig:param_study:toy}
    \end{center}
\end{figure}

\begin{figure}[t]
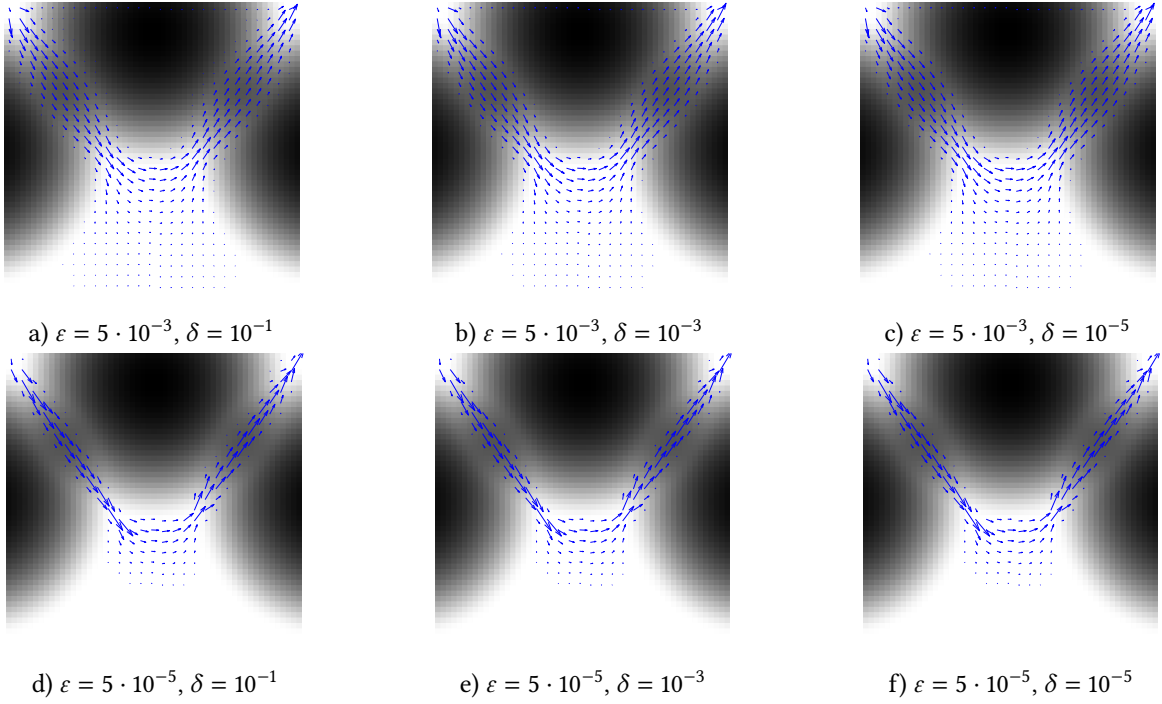

    \begin{center}
        \includeplot{valley/field_eps_5e-03_delta_1e-01.pdf}{\(\epsilon=5\cdot 10^{-3}\),
        \(\delta=10^{-1}\)}
        \includeplot{valley/field_eps_5e-03_delta_1e-03.pdf}{\(\epsilon=5\cdot 10^{-3}\),
        \(\delta=10^{-3}\)}
        \includeplot[0pt]{valley/field_eps_5e-03_delta_1e-05.pdf}{\(\epsilon=5\cdot 10^{-3}\),
        \(\delta=10^{-5}\)}\\
        \includeplot[\defaultfiguresep]{valley/field_eps_5e-05_delta_1e-01.pdf}{\(\epsilon=5\cdot
        10^{-5}\), \(\delta=10^{-1}\)}
        \includeplot{valley/field_eps_5e-05_delta_1e-03.pdf}{\(\epsilon=5\cdot 10^{-5}\),
        \(\delta=10^{-3}\)}
        \includeplot{valley/field_eps_5e-05_delta_1e-05.pdf}{\(\epsilon=5\cdot 10^{-5}\),
        \(\delta=10^{-5}\)}
        \caption{Visualization of the flow field for an example where the cost function is a
            mixture of three Gaussians.
            Both \(\mu^+\) and \(\mu^-\) are concentrated on a single square in
                the top left and top right corner, respectively.
        }\label{fig:param_study:valley}
    \end{center}
\end{figure}

\begin{figure}[t]
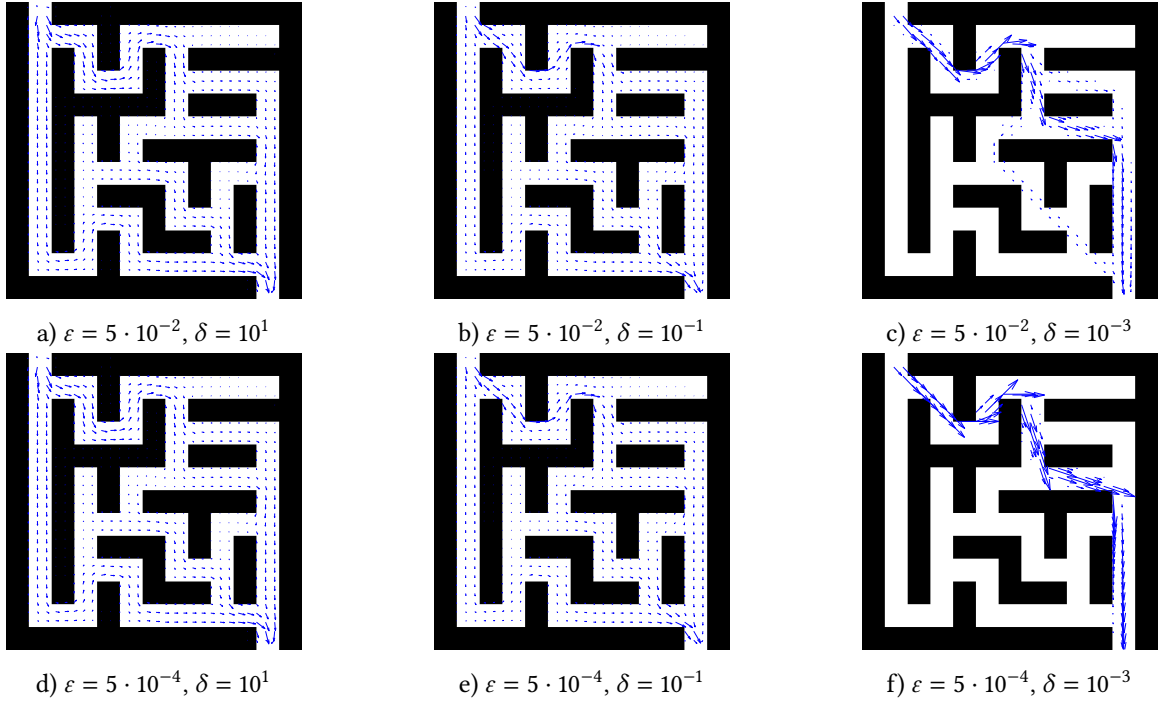

    \begin{center}
        \includeplot{maze/field_eps_5e-02_delta_1e+01.pdf}{\(\epsilon=5\cdot 10^{-2}\),
        \(\delta=10^{1}\)}
        \includeplot{maze/field_eps_5e-02_delta_1e-01.pdf}{\(\epsilon=5\cdot 10^{-2}\),
        \(\delta=10^{-1}\)}
        \includeplot[0pt]{maze/field_eps_5e-02_delta_1e-03.pdf}{\(\epsilon=5\cdot 10^{-2}\),
        \(\delta=10^{-3}\)}<fig:param_study:maze:c>\\
        \includeplot[\defaultfiguresep]{maze/field_eps_5e-04_delta_1e+01.pdf}{\(\epsilon=5\cdot
        10^{-4}\), \(\delta=10^{1}\)}
        \includeplot{maze/field_eps_5e-04_delta_1e-01.pdf}{\(\epsilon=5\cdot 10^{-4}\),
        \(\delta=10^{-1}\)}
        \includeplot[0pt]{maze/field_eps_5e-04_delta_1e-03.pdf}{\(\epsilon=5\cdot 10^{-4}\),
        \(\delta=10^{-3}\)}<fig:param_study:maze:f>
        \caption{Visualization of the flow field for an example where the cost function encodes a
            maze.
            Both \(\mu^+\) and \(\mu^-\) are concentrated on a single square in the top left and
                bottom
                right corner, respectively.
        }\label{fig:param_study:maze}
    \end{center}
\end{figure}

\Cref{fig:param_study:toy,fig:param_study:valley,fig:param_study:maze} show solutions
of~\cref{eq:double_reg_beckmann} for different choices of \(\mu^+\), \(\mu^-\) and \(w\).
The cost function \(w\) is encoded by the gray scale background, where darker shades denote higher
    costs.
In all cases, \(w\) is bounded away from zero.
The vector field \(q\) is encoded by the blue arrows.
For purposes of visualization, we display a downsampled version of \(q\), which was achieved by
    taking the average over the value of \(q\) across 4 squares (\ie 8 triangles) each.
Moreover, we only plot arrows who's Euclidian norm is larger than 1\% of the largest Euclidian
    norm of an entry in the averaged \(q\).
Note that the arrows are scaled for each subfigure independently.
The mesh consists of 5000 triangles for~\cref{fig:param_study:toy}, 6050 triangles
    for~\cref{fig:param_study:valley} and 8450 triangles for~\cref{fig:param_study:maze}.

We can observe that for \(\epsilon\to 0\), the solutions \(q\) become more singular, while for
    large \(\epsilon\) the regularization terms dominates the transportation cost so that the mass
    is
    transported more evenly through the domain.
As for the parameter \(\delta\) controlling the Huber regularization term, we can observe that
    while having only small influence on the regularity of \(q\), the overall objective value is
    reduced for large \(\delta\).
This can be seen best in~\cref{fig:param_study:maze}, where the maze has multiple solutions.
While for small \(\delta\) the shortest path is preferred, we see that other paths are used as
    well for larger \(\delta\).
This observation is in accordance to~\cref{thm:connection:primal} due to the terms \(-\delta w\)
    in \(\calG*\).

\subsection{Speed of Convergence}

\setlength{\defaultfiguresep}{0.1\textwidth}
\setlength{\figuresep}{0.1\textwidth}
\setlength{\figurewidth}{0.4\textwidth}
\begin{figure}[t]
    \begin{center}
        \includeplot{maze_errors_delta_1e-04.pdf}{Example from~\cref{fig:param_study:maze} for
            \(\epsilon=5\cdot 10^{-4}\)}[40mm 96mm 40mm 96mm]
        \includeplot[0pt]{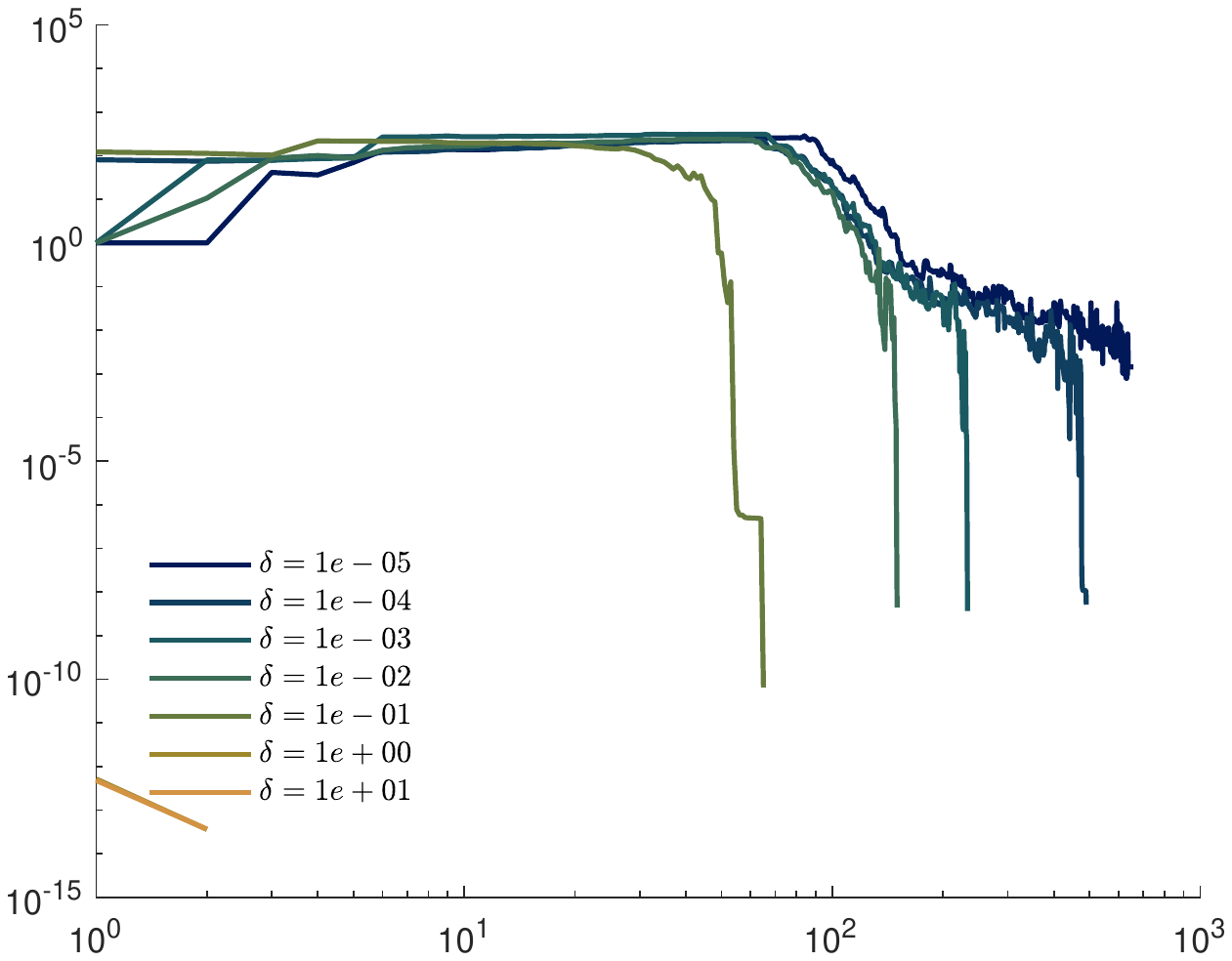}{Example from~\cref{fig:param_study:maze} for
        \(\delta=10^{-4}\)}[40mm 96mm 40mm 96mm]\\
        \includeplot[\defaultfiguresep]{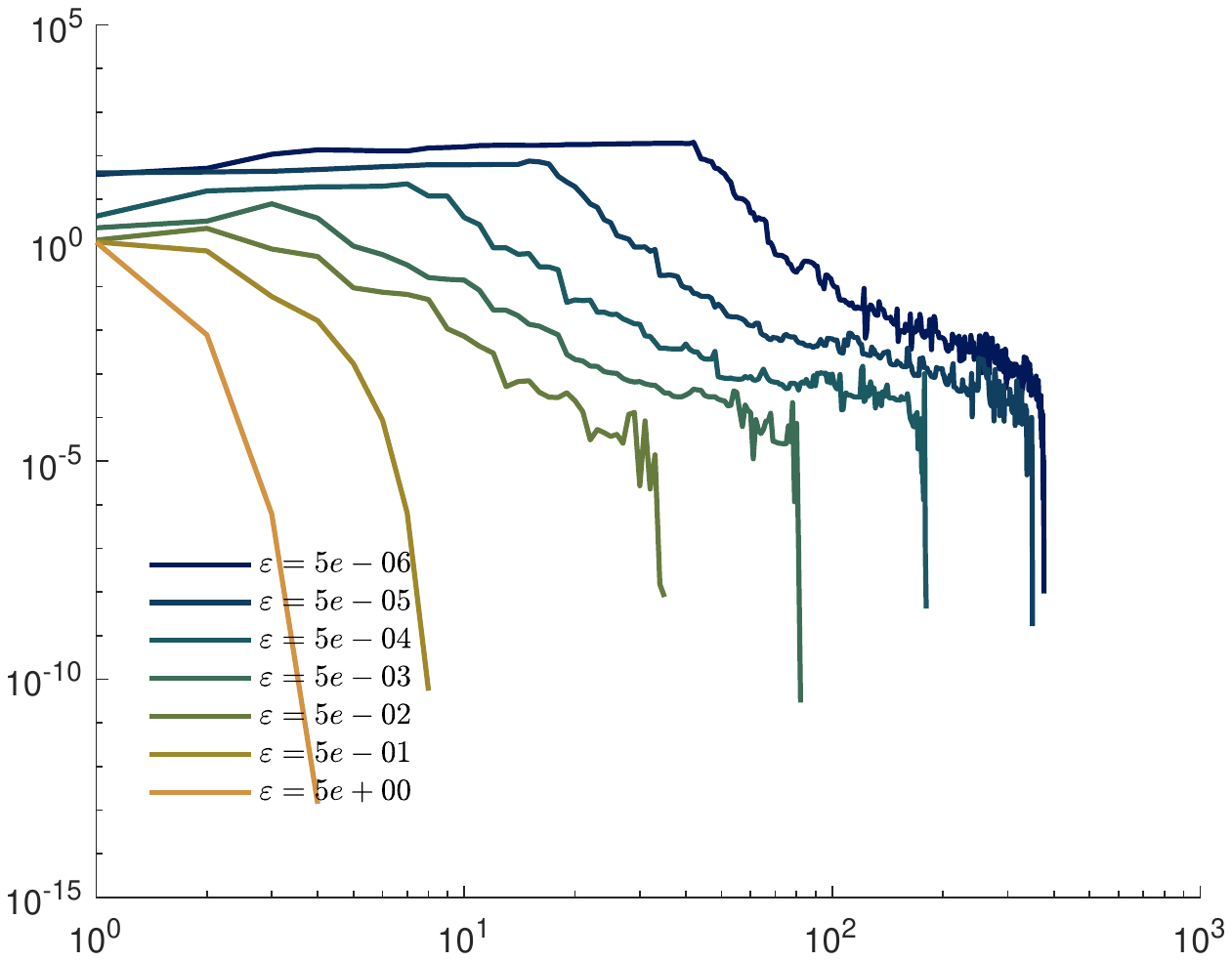}{Example
            from~\cref{fig:param_study:valley} for \(\epsilon=5\cdot 10^{-4}\)}[40mm 96mm 40mm
            96mm]
        \includeplot[0pt]{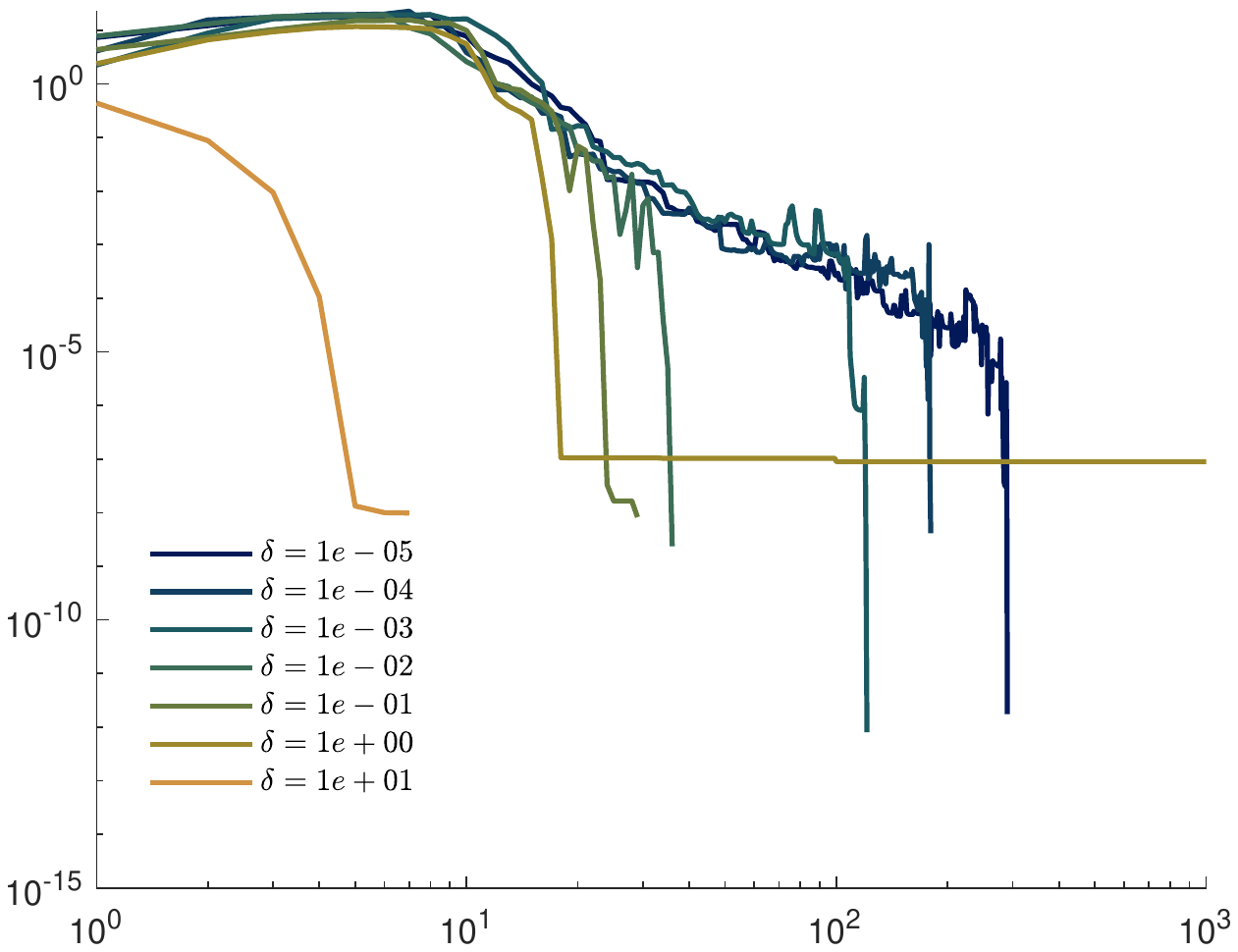}{Example from~\cref{fig:param_study:valley}
        for \(\delta=10^{-4}\)}[40mm 96mm 40mm 96mm]
        \caption{Relative error as in~\cref{eq:stopping_crit} for each iteration for selected
            instances of the examples shown
            in~\cref{fig:param_study:maze,fig:param_study:valley}.}\label{fig:convergence:error}
    \end{center}
\end{figure}

\Cref{fig:convergence:error} shows the observed relative errors in the optimality condition (as
described above) in dependence on the number of iterations for selected instances of the examples
from~\cref{fig:param_study:maze,fig:param_study:valley}.
We observe that larger regularization parameters, both for \(\delta\) and \(\epsilon\)
    significantly speed up convergence.
In fact, for some combinations of \(\delta\) and \(\epsilon\) the iteration failed to terminate
    for the given stopping criterion within the given maximum number of iterations.
These cases mostly
    correspond to very small regularization parameters.
However, for most test cases, we see quadratic convergence once we're close to the solution.

Note that larger values for \(\epsilon\) are interesting in the context of traffic
    congestion~\cite{Brasco_2010}.
The effect studied here can be observed by
    comparing~\cref{fig:param_study:maze:c,fig:param_study:maze:f}.
Here we can see, that the larger value of
    \(\epsilon\) promotes the shortest path, while the larger value promotes to spread the flow of
    mass across the different possible paths even if they are longer.

We also point out that our stopping criterion~\cref{eq:stopping_crit} is rather strict.
Among the
    literature reviewed in~\cref{sec:lit-review}, a similar criterion is used only
    in~\cite{benamou:2015,hatchi:2017}.
In these publications first order methods are employed, which
    naturally need a much higher number of iterations to achieve the same accuracy.
In~\cite{ryu:2017}
    a fixed-point residual of the Chambolle-Pock iteration is used as stopping criterion, which is
    not
    as easy to interpret.
For the ROF-Model in~\cite{jacobs:2019}, the authors are mainly interested
    in the objective value.
Hence, they only consider experiments where the objective value is known
    and use the error in the objective value as stopping criterion.
Finally,~\cite{barrett:2007,Facca_2020} use the relative change in the iterates \(\underline Q^j\)
    and \(\mu_h(t)\) (roughly corresponding to \(q\) and \(\abs{q}\) in our notation) as stopping
    criterion.

\section{Conclusion \& Outlook}\label{sec:conclusion}

In contrast to the original Beckmann problem, the \(\LL{\alpha}\)-regularized counterpart has
    unique solutions even for \(\mu^+, \mu^-\in\M\).
Moreover, this regularization naturally gives rise to a semi-smooth Newton scheme that can be used
    to solve the problem numerically.
For the iteration step to be well posed, we add a second regularization term of Huber type.
Convergence towards the original problem for vanishing regularization parameters can be proven, if
    the regularization parameters are coupled in an appropriate way.

This work can be extended both on the theoretic part and the numerical part.
On the theoretical
    part, a rigorous convergence theory for the proposed semi-smooth Newton
    iteration~\cref{alg:ssn}
    is still missing.
Regarding numerics, we have only worked with simple, fixed grids and similar
    to~\cite{barrett:2007} one could explore whether mesh adaption techniques are beneficial for
    the
    speed of convergence and accuracy of the solution.
Moreover one could employ path following schemes to try and improve the convergence speed.

\appendix
\pdfbookmark{Appendix}{appendix}
\section*{Appendix}
\stepcounter{section}

For the following result, \(\Tomega\) and \(\tilde\blank\) are defined as in~\cref{sec:approx}.

\begin{lemma}\label{thm:conv:vector_measure}
    Let \(0<\tau\to 0\) and let the notation of~\cref{sec:approx} hold.
    Let \(q\in\M*\) such that
        \(\mdiv q = \mu\) in \(\MMperp\) and denote \(q^\tau:=\phi_\tau*\tilde q\).
    Then
        \begin{enumerate}
        \item \(q^\tau \in \La[\Tomega]*\)
        \item\label{thm:conv:vector_measure:div} \(\div q^\tau = \phi_\tau*\tilde \mu\) in
        \(\Wperp[\Tomega]\)
        \item \(q^\tau \to \tilde q\) strongly in \(\M[\Tomega]*\).
    \item \(|q^\tau| \to \abs{\tilde q}\) strongly in \(\M[\Tomega]\) and hence also \wrt \ws
    convergence in \(\M[\Tomega]\).
    \end{enumerate}
\end{lemma}
\begin{proof}
    \begin{enumerate}
        \item See~\cite[Proposition 1.16]{Gaudry:2000}.
        \item Let \(\psi\in \Wempty[\Tomega]\).
              By~\cref{def:div} and Fubini's theorem, it holds
                  \begin{equation}\label{eq:divq:conv}
                  \begin{split}
                  -\dual{\div q^\tau}{\psi}
                  &= \intTO q^\tau \cdot \grad \psi\dleb
                  = \intTO\intTO \phi_\tau(y-x)\bd (\tilde q(x) \cdot \grad \psi(y)) \dy\\
                  &= \intTO \left(\intTO \phi_\tau(y-x) \grad \psi(y)\dy\right) \cdot
                  \bd \tilde q(x)\,.
              \\
              \end{split}
              \end{equation}
              Denote \(\psi_\tau(x) := \intTO \phi_\tau(y-x) \psi(y)\dy\).
              Then for every \(x\in\Omega\), integration by parts yields
                  \[
                      \intTO \phi_\tau(y-x) \grad \psi(y)\dy = \grad\psi_\tau(x)
                  \]
                  due to \(\phi_\tau(\blank-x) =0\) on \(\partial \Tomega\).
              Moreover, we note that \(\psi_\tau|_\Omega \in\W<1,\alpha'>\).
              Hence, since \(\tilde q\) is the extension by zero of \(q\) onto \(\Tomega\)
                  \begin{equation*}
                      -\dual{\div q^\tau}{\psi} = \intO \grad \psi_\tau \cdot\dq
                      = -\dual{\tilde \mu}{\psi_\tau}
                      = - \dual{\mu * \psi_\tau}{\psi}\,,
                  \end{equation*}
                  where the last equation follows analogously to~\cref{eq:divq:conv}.
        \item Let \(A\subset\Tomega\).
              Then
                  \begin{align*}
                  \abs{q^\tau - \tilde q}(A)
                  &= \abs[\big]{\int_A \intTO \phi_\tau(y-x) \bd \tilde q(x)
                      \dy -
                      \tilde q(A)}\\
                  &= \abs[\big]{\intTO \int_A \phi_\tau(y-x) \dy\bd \tilde q(x) - \tilde q(A)}\,.
              \end{align*}
              Clearly, the mapping \(x\mapsto \int_A \phi_\tau(y-x) \dy\)  is bounded by \(1\) for
                  all
                  \(\tau>0\) and converges to \(\1_{A}\) pointwise thanks to
                  \(\spt\tilde q\subset \Omega\).
              Hence, by dominated convergence, \(|q^\tau - \tilde q|(A)\to 0\).
        \item The last assertion is an immediate consequence of the reverse triangle
              inequality.\qedhere
    \end{enumerate}
\end{proof}

\begin{lemma}\label{thm:conv:vector_measure:spt}
    Let \(0<\tau\to 0\) and let the notation of~\cref{sec:approx} hold.
    Let \(q_\tau\in\M[\Omega_\tau]*\) and \(q\in\M[\Tomega]*\) such that \(\tilde q_\tau \wsto q\)
        in \(\M[\Tomega]*\).
    Then \(\spt q\subset \Omega\).
\end{lemma}
\begin{proof}
    Assume the contrary such that there is a Borel set \(A\subset \Tomega\setminus\Omega\) with
        \(\abs{q}(A) = M>0\).
    Then there is a compact set \(K\subset A\) such that \(\abs{q}(K) > \frac M2\) and
        \(\dist(K,\Omega) = r >0\).
    Set \(N:= K + \B{\frac r2}\).
    From the \ws lower semicontinuity of \(\M[\Tomega]*\ni p\mapsto
        \abs{p}(A)\in\RR\) for an arbitrary (relatively) open set \(A\) (c.f.
    e.g.~\cite{maggi_2012}), we deduce
    \begin{equation*}
        0 < \frac M2
        < \abs{\tilde q}(K)
        \leq \abs{\tilde q}(N)
        \leq \liminf_{\tau\to 0} \abs{\tilde q_\tau}(N) = 0\,,
    \end{equation*}
    where we used that \(\spt \tilde q_\tau\subset \Taumega\) and \(\Taumega\cap N = \emptyset\)
    for \(\tau>0\) sufficiently small.
\end{proof}

\section{Proof of
\texorpdfstring{\cref{thm:laplace_eps_continuous}}{%
Lemma~\ref{thm:laplace_eps_continuous}}}\label{sec:appdx:laplace_eps_cont}

    In order to derive a proof for~\cref{thm:laplace_eps_continuous}, we aim to express the
    transformation of the domain through a transformation of the differential operator.
First, we present a special case of~\cite[Theorem 1]{Groeger1989} that is adapted to our setting.

\begin{theorem}\label{thm:groeger_inverse}
    Let~\cref{assmpt:general,assmpt:laplace} hold with \(\alpha ' = r\geq 2\) and by \(r'\) denote
        the conjugate exponent, \ie \(\frac 1r + \frac 1{r'} =
        1\).
    Let \(D: \Omega\to\RR^{d\times d}\) be a measurable map satisfying \(mI \preceq D(x) \preceq
        MI\) for all
        \(x\in\Omega\) with \(0<m\leq M\).
    Define
        \begin{equation}\label{eq:defn_A}
            A: \Wr<\emptyset> \to \Wrperp\,,\qquad %
            \sca{Ay}{\phi} := \intO (D(x)\grad y(x))\cdot \grad
            \phi(x)\dx%
            \quad \forall \phi\in\Wrempty
        \end{equation}
        and \(\lapr := \div[\LL{r}]\grad: \W<1,r><\emptyset> \to \W<-1,r><\perp>\).
    Note that \(\lapr\) is continuously invertible by~\cref{assmpt:laplace}.
    Let \(m\) and \(M\) denote the infimum and supremum over \(x\in\Omega\) of the smallest and
        largest eigenvalue of \(D(x)\), respectively.
    Finally, set \(k := (1-\frac{m^2}{M^2})\).

    If \(k\norm{\lapr\inv} <1\), then \(A\) is bijective.
    Moreover, \(A\inv\) is continuous with
        \begin{equation*}
        \norm{A\inv}_{\Wrperp\to\Wr<\emptyset>} \leq
        \frac{m \norm{\lapr\inv}}{M^2(1 - k \norm{\lapr\inv})}\,.
    \end{equation*}
\end{theorem}
\begin{proof}
    This proof follows the outline of the proof given in~\cite[Theorem 1]{Groeger1989}.

    We first note that \(A\) is well defined and bounded as mapping from  \(\Wrempty\) to
        \(\Wrperp\), which can be seen by applying Hölder's inequality.
    Moreover, \(A\) is injective, which can be seen as follows.
    Let \(y_1,y_2\in\Wr<\emptyset>\) with \(Ay_1 = Ay_2\).
    Due to \(r\geq 2\), we may choose \(\phi = y_1 - y_2\in\Wr<\emptyset>\embed \Wrempty\)
        in~\cref{eq:defn_A},
        which yields
        \(0 = \intO (D (\grad y_1 - \grad y_2)) \cdot (\grad y_1 - \grad y_2)\).
    Because \(D\) has positive definite values, this implies \(\norm{y_1 - y_2} = 0\), as
        conjectured.

    Let now \(t:=m M^{-2}\) and let \(B: \L{r}* \to \L{r}*\), \((By)(x) := y(x) - t D(x)y(x)\).
    Clearly, \(B\) is linear.
    Moreover, \(B\) is bounded with \(\norm{B} \leq k\).

    Going on, let \(\nu\in\Wrm<\perp>\) and set
        \begin{equation}\label{eq:defn_q_mu}
        \begin{split}
        Q_\nu&:\Wr<\emptyset>\to\Wr<\emptyset>\,,\\
        {\sca{Q_\nu y}\phi}
        &= \sca{\lapr\inv(-\div[\LL{r}] B \grad y
            + t\nu)}{\phi}
        \quad\forall\phi\in\Wrpmperp\,.
    \end{split}
    \end{equation}
    Note that
        \begin{align}\label{eq:q_mu:reformulated}
        {\sca{Q_\nu y}\phi}
        & ={\sca y\phi}
        - t \sca{ D\grad y}{\grad {(\lapr\inv)}^*\phi}_{\L{r}\times\L{r'}}
        + t\sca{\nu}{{(\lapr\inv)}^*\phi}\nonumber\\
        &= \sca{y - t \lapr\inv (A y-\nu)}{\phi}\,.
    \end{align}

    From~\cref{eq:defn_q_mu}, it is straight forward to derive
        \begin{align*}
            \norm{Q_\nu\zeta - Q_\nu \xi}_{\Wr<\emptyset>}
             & \leq \norm{\lapr\inv(-\div[\LL{r}]) B \grad} \norm{\zeta - \xi}_{\Wr<\emptyset>}
            \\
             & \leq \norm{\lapr\inv}\norm{(-\div[\LL{r}]) B \grad} \norm{\zeta -
                \xi}_{\Wr<\emptyset>}
            \leq k \norm{\lapr\inv} \norm{\zeta - \xi}_{\Wr<\emptyset>}
        \end{align*}
        for all \(\zeta,\xi\in\Wr<\emptyset>\).
    Hence, \(Q_\nu\) is Lipschitz continuous with Lipschitz constant \( k \norm{\lapr\inv}\).
    Due to the assumption \(k \norm{\lapr\inv}<1\), \(Q_\nu\) is also strictly contractive and
        by~\cref{eq:q_mu:reformulated} the fixed point \(y\in\Wr\) of \(Q_\nu\) is a solution of
        \(Ay
        =
        \nu\).
    Hence, \(A\) is surjective
        and it remains to prove the conjectured continuity constant.

    To that end, let \(\nu, \rho\in\Wrm\) and let
        \(\xi,\zeta\in\Wr\)
        be the corresponding fixed points of \(Q_{\nu}\) and \(Q_\rho\).
    Then,
        \begin{align*}
        \norm{\zeta-\xi}_{\Wr}
        &= \norm{Q_\rho \zeta - Q_\nu \xi}_{\Wr}
        \leq  \norm{Q_\rho\zeta- Q_\rho \xi}_{\Wr} + \norm{Q_\rho\xi - Q_\nu \xi}_{\Wr}\\
        &\leq k \norm{\lapr\inv} \norm{\zeta - \xi}_{\Wr} + \norm{t \lapr\inv(\rho - \nu)}_{\Wr}\\
        &\leq k \norm{\lapr\inv} \norm{\zeta - \xi}_{\Wr} + t \norm{\lapr\inv} \norm{\rho -
            \nu}_{\Wrm}\,.
    \end{align*}
    Therefore, we obtain
        \begin{equation*}
            \norm{A\inv\rho - A\inv\nu}_{\Wr} (1 - k \norm{\lapr\inv})
            = \norm{\zeta -\xi }_{\Wr} (1 - k \norm{\lapr\inv})
            \leq t \norm{\lapr\inv} \norm{\rho - \nu}_{\Wrm}\,,
        \end{equation*}
        which concludes the proof.
\end{proof}

\Cref{thm:groeger_inverse} now allows us to solve the Poisson equation on \(\Taumega\), which is
covered by the following \namecref{thm:perturbed_laplace}.

\begin{lemma}\label{thm:perturbed_laplace}
    In the setting of~\cref{thm:groeger_inverse}, let~\cref{assmpt:starshaped} hold in addition
        and let \(\tau>0\).
    Then, the equation
        \begin{equation}\label{eq:stretching_on_emega}
        \int_{\Taumega} \grad y_\tau \cdot \grad \phi \dleb = \sca \phi {\nu_\tau}
        \quad \forall \phi\in \Wrempty[\Taumega]\,.
    \end{equation}
    has a unique solution \(y_\tau\in\W<1,r><\emptyset>[\Taumega]\) for every
    \(\nu_\tau\in\Wrperp\).
    Moreover, the solution operator
        \(\laprtau\inv:\Wrperp[\Taumega]\to\W<1,r><\emptyset>[\Taumega]\)
        of~\cref{eq:stretching_on_emega}
        is continuous with
        \begin{align*}
        \norm{\laprtau\inv}_{\Wrperp[\Taumega]\to\W<1,r><\emptyset>[\Taumega]}
        \leq \norm{\lapr\inv} {(1+\tau)}^2\,.
    \end{align*}
\end{lemma}
\begin{proof}
    \Wloss we assume \(\Taumega = (1+\tau)\Omega\).
    Let
        \(
        \Phi_\tau\phi(x) = \phi\of{\frac{x}{1+\tau}}
        \)
        and note that \(\Phi_\tau\) is
        a homeomorphism from \(\W<1,s><\emptyset>\) to	\(\W<1,s><\emptyset>[\Taumega]\)
        for every \(s>1\).
    Moreover, let \(\omega_\tau := {(1+\tau)}^{d-2}I \in \RR^{d\times d}\).
    Note that the only
        eigenvalue of \(\omega_\tau\) is \({(1+\tau)}^{d-2}\).
    Let now \(g\in \Wrperp\) and consider the equation
        \begin{equation}\label{eq:stretching_starting_point}
        \intO \omega_\tau \grad y \cdot \grad v \dleb = \sca gv
        \quad \forall v\in \Wrempty\,.
    \end{equation}
    By~\cref{thm:groeger_inverse},~\cref{eq:stretching_starting_point} has a unique solution
        \(y\in\W<1,r><\emptyset>\).
    By defining \(g\in \Wrperp\) via
        \(g := {(1+\tau)}^d\Phi_\tau\inv \nu_{\tau}\)
        as well as \(y_\tau := \Phi_\tau y\in\W<1,r><\emptyset>[\Taumega]\) and inserting both
        into~\cref{eq:stretching_starting_point}, we obtain
        \begin{align}
            \intO \of{\omega_\tau \grad \Phi_\tau\inv y_\tau} \cdot \grad \Phi_\tau\inv\psi
             & = \sca\psi{\nu_\tau}
            \quad \forall \psi\in\Wrempty[\Taumega]\,,\label{eq:haller_dintelmann}
        \end{align}
        where we have used that \(\Phi_\tau\) is a bijection.
    Using the transformation formula,~\cref{eq:haller_dintelmann} can be seen to be
        equivalent
        to~\cref{eq:stretching_on_emega} and hence, \(y_\tau\) is a solution
        of~\cref{eq:stretching_on_emega}.
    Note that \(y_\tau\) is the unique solution, since \(y\) is the unique solution
        of~\cref{eq:stretching_starting_point} and \(\Phi_\tau\) is a bijection.

    To show continuity of the solution operator, we first note that thanks
        to~\cref{eq:stretching_starting_point} for \(y_\tau\) it holds
        \begin{equation*}
            \norm{y_\tau}_{\W<1,r><\emptyset>[\Taumega]}
            = \omega_\tau\inv \norm{\Phi_\tau\lapr\inv g}_{\W<1,r><\emptyset>[\Taumega]}
            \leq \frac{
            \norm{\Phi_\tau}_{\W<1,r><\emptyset> \to \W<1,r><\emptyset>[\Taumega]}
            \norm{\lapr\inv}
            \norm{g}_{\Wrperp}
            }{
            {(1+\tau)}^{d-2}
            } \,,
        \end{equation*}
        where
        \begin{align*}
            \norm{g}_{\Wrperp}
            \leq \sup_{1=\norm{v}_{\Wrempty}}
            \norm{\nu_\tau}_{\Wrperp[\Taumega]} \norm{\Phi_\tau v}_{\Wrempty[\Taumega]}
            \leq \norm{\nu_\tau}_{\Wrperp[\Taumega]}
            \norm{\Phi_\tau}_{\Wrempty\to\Wrempty[\Taumega]}
        \end{align*}
        and it remains to compute the operator norm
        \(\norm{\Phi_\tau}_{\W<1,s><\emptyset>\to\W<1,s><\emptyset>[\Taumega]}\).
    To this end, let \(\phi\in\W<1,s><\emptyset>\).
    Using the transformation formula, it is
        straightforward to compute
        \begin{align*}
            \norm{\Phi_\tau \phi}_{\L{s}[\Taumega]}
             & = {(1+\tau)}^{\frac ds} \norm{\phi}_{\L{s}}\,,            \\
            \norm{\frac 1{1+\tau} \grad \Phi_\tau\phi}_{\L{s}[\Taumega]}
             & = {(1+\tau)}^{\frac ds - 1} \norm{\grad \phi}_{\L{s}*}\,,
        \end{align*}
        such that
        \begin{align*}
        \norm{\Phi_\tau}_{\W<1,s><\emptyset>\to\W<1,s><\emptyset>[\Taumega]}
        &= \sup_{1=\norm{\phi}_{_{\W<1,s><\emptyset>}}} {(1+\tau)}^{\frac ds} \norm{\phi}_{\L{s}}
        +
        {(1+\tau)}^{\frac ds - 1} \norm{\grad \phi}_{\L{s}*}\\
        &\leq \sup_{\substack{\zeta,\xi\in[0,1]\\\zeta+\xi=1}} \zeta {(1+\tau)}^{\frac ds} + \xi
        {(1+\tau)}^{\frac ds - 1}
        = {(1+\tau)}^{\frac ds}\,.
    \end{align*}
    Hence,
        \begin{align*}
            \frac{\norm{\laprtau\inv \nu_{\tau}}_{\W<1,r><\emptyset>[\Taumega]}
            }{\norm{\nu_\tau}_{\Wrperp[\Taumega]}}
            = \frac{\norm{y_\tau}_{\W<1,r><\emptyset>[\Taumega]}
            } {\norm{\nu_\tau}_{\Wrperp[\Taumega]}}
            \leq \frac{\norm{\lapr\inv}}{{(1+\tau)}^{d-2}}
            {(1+\tau)}^{\frac{d}r}{(1+\tau)}^{\frac{d}{r'}}
            = \norm{\lapr\inv} {(1+\tau)}^2\,,
        \end{align*}
        which yields \(\norm{\laprtau\inv}_{\Wrperp[\Taumega]\to\W<1,r><\emptyset>[\Taumega]} \leq
        \norm{\lapr\inv} {(1+\tau)}^2 \)
        and concludes the proof.
\end{proof}

Finally, we're in the position to prove~\cref{thm:laplace_eps_continuous}.

\begin{proof}
    Choosing \(r=\alpha'\) in~\cref{thm:perturbed_laplace}, we obtain that
        \begin{equation*}
            - \int_{\Taumega} \grad \zeta \cdot \grad \phi = \int_{\Taumega} \phi \bd \xi \quad
            \forall \phi \in\W<1,\alpha><\emptyset>[\Taumega]\,,
        \end{equation*}
        has a unique solution \(\zeta\in\Wempty[\Taumega]\) for all
        \(\xi\in\W<-1,\alpha'><\perp>[\Taumega]\),
        which corresponds to~\cref{eq:poisson} on \(\Taumega\).
    Moreover, the corresponding solution
        operator is uniformly bounded for \(\tau\to 0\).
    The assertion now follows analogously to the proof of~\cref{thm:div:surjective}.
\end{proof}

\bibliographystyle{plain}
\bibliography{refs}

\end{document}